\documentclass[11pt]{amsart}
\usepackage[dvips]{color}
\usepackage{amsmath}
\usepackage{amsxtra}
\usepackage{amscd}
\usepackage{amsthm}
\usepackage{amsfonts}
\usepackage{amssymb}
\usepackage{eucal}
\usepackage{epsfig}

\usepackage{young}
\usepackage[vcentermath]{youngtab}
\theoremstyle{plain}
\newtheorem{theorem}{Theorem}[section]

\newtheorem{cor}[theorem]{Corollary}
\newtheorem{prop}[theorem]{Proposition}
\newtheorem{lem}[theorem]{Lemma}

\newtheorem{defi}[theorem]{Definition}
\theoremstyle{definition}

\newtheorem{example}{Example}
\newtheorem{question}{Question}
\newtheorem{rem}[theorem]{Remark}

\newcommand{\Glie}{\mathfrak{g}}             
\newcommand{\Gaff}{\widehat{\mathfrak{g}}}   
\newcommand{\Gafft}{\widehat{\mathfrak{g}'}} 
\newcommand{\ev}{\textrm{ev}}                
\newcommand{\CR}{\textbf{R}}                 

\newcommand{\BC}{\mathbb{C}}            
\newcommand{\BZ}{\mathbb{Z}}            

\newcommand{\BV}{\mathbf{V}}             
\newcommand{\BW}{\mathbf{W}}           
\newcommand{\CW}{\mathcal{W}}       
\newcommand{\BP}{\mathbf{P}}             
\newcommand{\End}{\mathrm{End}}          
\newcommand{\Id}{\textrm{Id}}            
\newcommand{\Sm}{\mathbb{S}}             

\newcommand{\super}{\mathbb{Z}_2}        
\newcommand{\even}{\overline{0}}         
\newcommand{\odd}{\overline{1}}          

\allowdisplaybreaks                

\begin{document}
\begin{title}[Quantum affine superalgebras]
{Fundamental representations of quantum affine superalgebras and $R$--matrices}
\end{title}
\author{Huafeng Zhang}
\address{Departement Mathematik, ETH Z\"{u}rich, CH-8092 Z\"{u}rich, Switzerland.} 
\address{Institut f\"ur Theoretische Physik, ETH Z\"urich, CH-8093 Z\"urich, Switzerland.}
\email{huafeng.zhang@math.ethz.ch} 
\begin{abstract}
We study a certain family of finite-dimensional simple representations over quantum affine superalgebras associated to general linear Lie superalgebras, the so-called fundamental representations: the denominators of rational $R$--matrices between two fundamental representations are computed; a cyclicity (and so simplicity) condition on tensor products of fundamental representations is proved. 
\end{abstract}

\maketitle
\setcounter{tocdepth}{1}
\section{Introduction}
Fix $M,N$ two natural numbers and $q$ a non-zero complex number which is not a root of unity.
Let $\Glie := \mathfrak{gl}(M,N)$ be the {\it general linear Lie superalgebra}. Let $U_q(\Gaff)$ be the associated quantum affine superalgebra. This is a Hopf superalgebra neither commutative nor co-commutative, and it can be seen as a deformation of the universal enveloping algebra of the affine Lie superalgebra $L \Glie := \Glie \otimes \BC[t,t^{-1}]$.
In this paper we are mainly concerned with the structure of tensor products of finite-dimensional simple $U_q(\Gaff)$-modules.

Quantum affine superalgebras, as supersymmetric generalization of quantum affine algebras, were defined previously by Yamane \cite{Y} with Drinfeld--Jimbo generators (and with Drinfeld loop generators for $U_q(\Gaff)$). They appeared as the algebraic supersymmetries of solvable models such as the $q$--state vertex model \cite{Perk-Schultz} and the $t$--$J$ models \cite{Kojima}; their highest weight representations were identified in these models with the spaces of states to compute correlation functions. Recently, various quantum superalgebras (finite type, affine type, Yangian) together with their finite-dimensional representations associated to the simple Lie superalgebra $\mathfrak{psl}(2,2)$ are linked to the integrability structures in the context of the AdS/CFT correspondence and in Hubbard model (see \cite{Beisert} and its references). Quantum affine superalgebra associated to the exceptional Lie superalgebra $D(2,1;x)$ is related to generalized hypergeometric equations \cite{BL}.  

Compared to the rich literature on quantum affine algebras (see the review papers \cite{CH,L}), quantum affine superalgebras have been less studied. Technical difficulties already arise in the situation of finite-dimensional simple Lie superalgebras: all the Borel subalgebras are not conjugated, Weyl groups are not enough to characterize linkage, etc. 

The series of papers \cite{Z1,Z2,Z3} studied systematically finite-dimensional representations of $U_q(\Gaff)$. In \cite{Z1}, there is a similar highest $\ell$-weight classification \cite{CP2} of finite-dimensional simple modules adapted to the Drinfeld new realization of $U_q(\Gaff)$. Our motivating questions are as follows. Let $S_1,S_2,\cdots,S_n$ be such $U_q(\Gaff)$-modules.
\begin{itemize}
\item[(I)] Construct $U_q(\Gaff)$-module morphisms from $S_1 \otimes S_2$ to $S_2 \otimes S_1$. 
\item[(II)] Determine when $S_1 \otimes S_2 \otimes \cdots \otimes S_n$ is a highest $\ell$-weight module.
\end{itemize}
In the non-graded case, (I) and (II) are related to each other by the notion of {\it normalized $R$-matrix} $R_{S_1,S_2}$ proposed in \cite{AK}. This is a matrix-valued rational function depending on the ratio $\frac{a}{b}$ of spectral parameters $a,b \in \BC^{\times}$ of $S_1$ and $S_2$ respectively. Whenever it is well-defined (in other words the denominator of $R_{S_1,S_2}$ is non-zero when specialized to $S_1,S_2$), $R_{S_1,S_2}$ composed with the flip map is a module morphism from $S_1 \otimes S_2$ to $S_2 \otimes S_1$. It was proved in \cite{Kashiwara} (first conjectured in \cite{AK}) that the tensor product in (II) is of highest $\ell$-weight if  $R_{S_i,S_j}$ is well-defined for all $i < j$, under the assumption that the $S_i$ are {\it good} modules within the framework of crystal base theory. Similar results were obtained by Varagnolo--Vasserot \cite{VV} for fundamental modules over simply-laced quantum affine algebras via Nakajima quiver varieties, and by Chari \cite{Chari} in general situations via the braid group action on affine Cartan subalgebras. Here fundamental modules are certain simple modules whose highest $\ell$-weights are of particular forms \cite[Definition 3.4]{CH}.

Quite recently, normalized $R$-matrices were used to establish generalized Schur--Weyl duality between representations of quantum affine algebras and those of quiver Hecke algebras and monoidal categorifications of (quantum) cluster algebras \cite{HL,KKK,KKKO}. We refer to the table in \cite[Appendix A]{Oh} for a summary of p\^{o}les with multiplicity of normalized $R$-matrices between two fundamental modules over quantum affine algebras. We mention earlier works of Chari--Pressley \cite{CP1} on zeros and p\^{o}les of $R$-matrices for Yangians.

In this paper, we study (I) and (II) for fundamental modules over $U_q(\Gaff)$. The fundamental modules $V_{r,a}^{\varepsilon}$ over $U_q(\Gaff)$ are defined by a fusion procedure (Definition \ref{def: fundamental representations}). They depend on a spectral parameter $a \in \BC^{\times}$ and a Dynkin node together with signature $(r,\varepsilon)$: positive if $(\varepsilon = +, 1\leq r \leq M)$ and negative if $(\varepsilon = -, 1\leq r \leq N)$. When $N=0$ they are the fundamental modules in the non-graded case \cite{Date}. The main results of this paper are:
\begin{itemize}
\item[(A)] denominators of $R$-matrices between two fundamental modules (Theorems \ref{thm: pole even odd}-\ref{thm: even even poles});
\item[(B)] a sufficient condition for a tensor product of fundamental modules $\otimes_{i=1}^n V_{r_i,a_i}^{\varepsilon_i}$ to be of highest $\ell$-weight when $(\varepsilon_1\varepsilon_2\cdots \varepsilon_n) = (++\cdots + --\cdots -)$ (Theorem \ref{thm: main result}).
\end{itemize} 
In \S \ref{sec: discussion} we indicate a general idea to study tensor products of arbitrary signatures. Eventually it is enough to solve a problem of linear algebra (Question \ref{question: linear algebra}). (B) has the following two consequences. Let $S_1,S_2,\cdots,S_n$ be fundamental modules.
\begin{itemize}
\item[(C)] $S_1 \otimes S_2 \otimes \cdots \otimes S_n$ is simple if and only if so is $S_i \otimes S_j$ for all $i < j$ (Theorem \ref{thm: simplicity of tensor products}).
\item[(D)] if the parities of the $S_i$ are the same, then $S_1 \otimes S_2 \otimes \cdots \otimes S_n$ is simple if and only if it is of highest $\ell$-weight and of lowest $\ell$-weight (Corollary \ref{cor: tensor product even fundamental}).
\end{itemize}

Let us make comparisons of (A)--(D) with related results in literature. 

1. When restricted to the finite type quantum superalgebra $U_q(\Glie)$, a positive (resp. negative) fundamental module is in the category $\mathcal{O}$ (resp. its dual category $\mathcal{O}^*$) of \cite{BKK}, and their tensor products may not be semi-simple. In deducing (A) we make tricky use of a fact (Lemma \ref{lem: cyclicity Chari}) on the tensor product of a highest $\ell$-weight module and a lowest $\ell$-weight module, as opposed to the non-graded case \cite{CP1,Date,KOS,Oh} where tensor product decompositions and spectral decompositions were usually needed. 

Our arguments can be applied to the non-graded case. However it seems that even in the situation \cite{Date} of quantum affine algebras of type A the calculations would become more involved than those in Theorems \ref{thm: pole even odd}--\ref{thm: even even poles}. By the fusion procedure lowest $\ell$-weight vectors of fundamental modules are pure tensors in our situation, while they are alternating sums over symmetric groups in \cite{Date}. The denominators in Theorems \ref{thm: pole even odd}--\ref{thm: even even poles} are simpler than those in \cite[Equation (2.8)]{Date}. Notably, if $S_1$ and $S_2$ are fundamental modules of different signatures, then the denominator of $R_{S_1,S_2}$ is a polynomial of degree 1.

We expect similar simplification of denominators of $R_{S_1,S_2}$ for more general simple $U_q(\Gaff)$-modules $S_1$ in the category $\mathcal{O}$ and $S_2$ in the category $\mathcal{O}^*$. This might be related to the crystal base theory developed in \cite{BKK} for $U_q(\Glie)$-modules. 

\medskip 

2. (B) can be viewed as a super version of cyclicity results in \cite{Chari,Kashiwara,VV}.  As explained in the introduction of \cite{Z2}, the methods in the non-graded case do not admit straightforward generalizations.  Nevertheless a weaker result has been proved in \cite{Z2} under the assumption that in the tensor product of (B) the $(r_i,\varepsilon_i)$ must be the same. This weaker result has been used in \cite{Z3} to construct asymptotic modules in the sense of Hernandez--Jimbo \cite{HJ}, and it will again be needed in the present paper to validate the fusion procedure (in the proof of Proposition \ref{prop: simplicity of fundamental modules}). 

The idea of proof in \cite{Z2} is a modification of Chari's reduction arguments in \cite{Chari}: to restrict $U_q(\Gaff)$-modules to $U_q(\widehat{\mathfrak{gl}(1,1)})$-modules. Every step of reduction therein resulted in tensor products of two-dimensional simple modules. An essential improvement in this paper is to view these tensor products as Weyl modules (Lemma \ref{lem: reduction Weyl gl(1,1)}). From this viewpoint the reduction arguments in \cite{Z2} work equally well even if the $(r_i,\varepsilon_i)$ change. 

The Weyl modules over $U_q(\Gaff)$ were defined in \cite{Z1}; they are super analog of Weyl modules over quantum affine algebras \cite{CP3}. The case of $\mathfrak{gl}(1,1)$ is already useful enough to prove (B). It would be interesting to look at general case of $\mathfrak{gl}(M,N)$. 

We mention the recent works of Guay--Tan \cite{GT} on a similar cyclicity result where the braid group action in \cite{Chari} was defined for Yangians. It is question to construct similar braid group (or groupoid) action in the super case in order to study more general simple modules. For this, it might be useful to look at different RTT realizations of $U_q(\Gaff)$ (by permuting the parity of the base vectors in $\BV$ \cite[Definition 3.5]{Z2}).

\medskip

3. (C) is true for all finite-dimensional simple modules over quantum affine algebras. Its proof in \cite{H} utilized deep theory of $q$-characters of Frenkel--Reshetikhin. In our situation, since we are restricted to fundamental modules, up to some duality arguments, (C) is a direct consequence of (B).

 (D) is special in the super case, and has been proved in \cite{Z2} for all finite-dimensional simple modules over $U_q(\widehat{\mathfrak{gl}(1,1)})$. In the non-graded case, due to the action of Weyl groups, such a tensor product is of highest $\ell$-weight if and only if it is of lowest $\ell$-weight. 
 
The paper is organized as follows. \S \ref{sec: pre} prepares the necessary background on highest $\ell$-weight modules and on fundamental modules. \S \ref{sec: spin r matrix} constructs the normalized $R$-matrices between two fundamental modules from elementary ones via a fusion procedure. \S \ref{sec: even odd} computes the denominators of the normalized $R$-matrices. \S \ref{sec: Weyl modules} proves some easy but important properties of Weyl modules. \S \ref{sec: cyclicity} proves (B) by a series of reductions. \S \ref{sec: simplicity} then discusses the consequences of (B). \S \ref{sec: discussion} reduces the general case of (B) without assumption on signature to a question of linear algebra (Question \ref{question: linear algebra}).

\subsection*{Acknowledgments.} The author is grateful to David Hernandez and to Kenji Iohara and Bernard Leclerc for interesting discussions. He thanks Masato Okado for sending a reprint of \cite{Date}. He thanks the anonymous referees for useful comments and suggestions.
\section{Preliminaries}   \label{sec: pre}
Fix $M,N \in \BZ_{>0}$. This section collects basic facts on quantum superalgebras associated to the general linear Lie superalgebra $\mathfrak{gl}(M,N)$ and their representations.  

\subsection{Quantum superalgebras.}\label{subsec: quantum superalgebra}  Set $\kappa := M+N$, $I := \{1,2,\cdots,\kappa \}$ and 
\begin{displaymath}
|\cdot|: I \longrightarrow \super, i \mapsto |i| =: \begin{cases}
\even & (i\leq M),  \\
\odd & (i > M),
\end{cases} \quad d_{\cdot}: I \longrightarrow \BZ, i \mapsto d_i := \begin{cases}
1 & (i \leq M),\\
-1 & (i > M).
\end{cases}
\end{displaymath}
Set $q_i := q^{d_i}$. Set $\BP := \oplus_{i \in I} \BZ \epsilon_i$. Let $(,): \BP \times \BP \longrightarrow \BZ$ be the bilinear form defined by $(\epsilon_i,\epsilon_j) = \delta_{ij}d_i$. Let $|\cdot|: \BP \longrightarrow \super$ be the morphism of abelian groups such that $|\epsilon_i| = |i|$.

In the following, we only consider the parity $|x| \in \super$ of $x$ when either $x \in I, x \in \BP$ or $x$ is a $\super$-homogeneous vector of a vector superspace. Associated to two vector superspaces $V$ and $W$ is the graded permutation $c_{V,W}: V \otimes W \longrightarrow W \otimes V$ defined by $v \otimes w \mapsto (-1)^{|v||w|} w\otimes v$. Except in \S \ref{sec: Weyl modules}, $\Glie$ always denotes $\mathfrak{gl}(M,N)$, while $\Glie' = \mathfrak{gl}(N,M)$.

Let $\BV = \oplus_{i\in I} \BC v_i$ be the vector superspace with $\super$-grading $|v_i| = |i|$. For $i,j \in I$, let $E_{ij} \in \End(\BV)$ be the endomorphism $v_k \mapsto \delta_{jk} v_i$. Introduce the Perk--Schultz matrix $R(z,w) \in \End(\BV^{\otimes 2})[z,w]$:
\begin{eqnarray}  \label{for: Perk-Schultz matrix coefficients}
\begin{array}{rcl}
R(z,w) &=&  \sum\limits_{i\in I}(zq_i - wq_i^{-1}) E_{ii} \otimes E_{ii}  + (z-w) \sum\limits_{i \neq j} E_{ii} \otimes E_{jj} \\
&\ & + z \sum\limits_{i<j} (q_i-q_i^{-1}) E_{ji} \otimes E_{ij} + w \sum\limits_{i<j}(q_j-q_j^{-1})  E_{ij} \otimes E_{ji}.
\end{array}  
\end{eqnarray}
It is well-known that $R(z,w)$ satisfies the quantum Yang--Baxter equation:
$$R_{12}(z_1,z_2)R_{13}(z_1,z_3) R_{23}(z_2,z_3) = R_{23}(z_2,z_3)R_{13}(z_1,z_3)R_{12}(z_1,z_2). $$
Here we use the following convention for the tensor subscripts. Let $n \geq 2$ and $A_1,A_2,\cdots,A_n$ be unital superalgebras. Let $1 \leq i < j \leq n$. If $x \in A_i$ and $y \in A_j$, then 
\begin{displaymath}
(x\otimes y)_{ij} := (\otimes_{k=1}^{i-1} 1_{A_k}) \otimes x \otimes (\otimes_{k=i+1}^{j-1} 1_{A_k}) \otimes y \otimes (\otimes_{k=j+1}^n 1_{A_k}) \in \otimes_{k=1}^n A_k.
\end{displaymath}
Now we can define the quantum affine superalgebra associated to $\Glie$.
\begin{defi}\cite{Z2}      \label{def: quantum affine superalgebras}
The quantum affine superalgebra $U_q(\Gaff)$ is the superalgebra defined by 
\begin{itemize}
\item[(R1)] RTT-generators $s_{ij}^{(n)}, t_{ij}^{(n)}$ for $i,j \in I$ and $n \in \BZ_{\geq 0}$;
\item[(R2)] $\super$-grading $|s_{ij}^{(n)}| = |t_{ij}^{(n)}| = |i| + |j|$; 
\item[(R3)] RTT-relations in $U_q(\Gaff) \otimes (\End \BV^{\otimes 2})[[z,z^{-1},w,w^{-1}]]$
\begin{eqnarray*}
&& R_{23}(z,w) T_{12}(z) T_{13}(w) = T_{13}(w) T_{12}(z) R_{23}(z,w),  \\
&& R_{23}(z,w) S_{12}(z) S_{13}(w) = S_{13}(w) S_{12}(z) R_{23}(z,w),    \\
&& R_{23}(z,w) T_{12}(z) S_{13}(w) = S_{13}(w) T_{12}(z) R_{23}(z,w),    \\
&& t_{ij}^{(0)} = s_{ji}^{(0)} = 0 \quad \mathrm{for}\ 1 \leq i < j \leq \kappa,    \\
&& t_{ii}^{(0)} s_{ii}^{(0)} = 1 = s_{ii}^{(0)} t_{ii}^{(0)} \quad \mathrm{for}\ i \in I.  
\end{eqnarray*}
\end{itemize}
Here $T(z) = \sum_{i,j \in I} t_{ij}(z) \otimes E_{ij} \in (U_q(\Gaff) \otimes \End \BV)[[z^{-1}]]$ and $t_{ij}(z) = \sum_{n \in \BZ_{\geq 0}} t_{ij}^{(n)} z^{-n} \in U_q(\Gaff)[[z^{-1}]]$, and similar definition of $S(z)$ except that the $z^{-n}$ is replaced by the $z^{n}$.
\end{defi}
$U_q(\Gaff)$ has a Hopf superalgebra structure with counit $\varepsilon: U_q(\Gaff) \longrightarrow \BC$ defined by $\varepsilon(s_{ij}^{(n)}) = \varepsilon(t_{ij}^{(n)}) = \delta_{ij}\delta_{n0}$, and coproduct $\Delta: U_q(\Gaff) \longrightarrow U_q(\Gaff)^{\otimes 2}$:
\begin{eqnarray}   
&&\Delta (s_{ij}^{(n)}) = \sum_{m=0}^n \sum_{k \in I} \epsilon_{ijk}  s_{ik}^{(m)} \otimes s_{kj}^{(n-m)}, \quad  \Delta (t_{ij}^{(n)}) = \sum_{m=0}^n \sum_{k \in I} \epsilon_{ijk} t_{ik}^{(m)} \otimes t_{kj}^{(n-m)}. \label{for: coproduct for quantum affine superalgebra S}   
\end{eqnarray}
Here $\epsilon_{ijk} := (-1)^{(|i|+|k|)(|k|+|j|)}$. The antipode $\Sm: U_q(\Gaff) \longrightarrow U_q(\Gaff)$ is determined by
\begin{eqnarray}  \label{for: antipode}
&&(\Sm \otimes \Id)(S(z)) = S(z)^{-1}, \quad (\Sm \otimes \Id)(T(z)) = T(z)^{-1}.    
\end{eqnarray}
Here the RHS of the above formulas are well-defined owing to the last two relations in Definition \ref{def: quantum affine superalgebras}.
 The subalgebra of $U_q(\Gaff)$ generated by the $s_{ij}^{(0)},t_{ij}^{(0)}$ is a sub-Hopf-superalgebra denoted by $U_q(\Glie)$. To simplify notations, write $s_{ij} := s_{ij}^{(0)}, t_{ij} := t_{ij}^{(0)}$.

We recall symmetry properties of $U_q(\Gaff)$, following mainly \cite{Z2,Z3}.

For $\mathfrak{gl}(N,M) =: \Glie'$, let us define the quantum superalgebras $U_q(\Gafft),U_q(\Glie')$ in exactly the same way as $U_q(\Gaff), U_q(\Glie)$, except that we interchange $M,N$ everywhere. Let $s_{ij}'^{(n)},t_{ij}'^{(n)}$ for $i,j \in I$ and $n \in \BZ_{\geq 0}$ be the corresponding RTT generators of $U_q(\Gafft)$, so that their $\super$-degrees are $|s_{ij}'^{(n)}| = |t_{ij}'^{(n)}| = |i|' + |j|'$ where $|i|' = \even$ for $1 \leq i \leq N$ and $\odd$ otherwise. For $i \in I$, set $\widehat{i} := \kappa+1-i$.  Let $a \in \BC^{\times}$. The following are isomorphisms of Hopf superalgebras ($\varepsilon_{ij} := (-1)^{|i|+|i||j|}$ and $\varepsilon_{ij}' := (-1)^{|i|'+|i|'|j|'}$)
\begin{eqnarray}    
&& \Phi_a: U_q(\Gaff) \longrightarrow U_q(\Gaff),\quad s_{ij}^{(n)} \mapsto a^n s_{ij}^{(n)},\ t_{ij}^{(n)} \mapsto a^{-n} t_{ij}^{(n)},   \label{for: automorphisms of Z-graded superalgebras} \\
&& \Psi: U_q(\Gaff) \longrightarrow U_q(\Gaff)^{\mathrm{cop}},\quad s_{ij}^{(n)} \mapsto \varepsilon_{ji}t_{ji}^{(n)},\quad t_{ij}^{(n)} \mapsto \varepsilon_{ji}s_{ji}^{(n)},   \label{equ: transposition in quantum affine superalgebra}   \\
&& f: U_q(\Gafft) \longrightarrow U_q(\Gaff)^{\mathrm{cop}},\quad s_{ij}'^{(n)} \mapsto \varepsilon_{ji}' s_{\widehat{j}\widehat{i}}^{(n)},\quad t_{ij}'^{(n)} \mapsto \varepsilon_{ji}' t_{\widehat{j}\widehat{i}}^{(n)}.    \label{equ: from gl(M,N) to gl(N,M)}
\end{eqnarray}  
Here $A^{\mathrm{cop}}$ of a Hopf superalgebra $A$ takes the same underlying superalgebra but the twisted coproduct $\Delta^{\mathrm{cop}} := c_{A,A} \Delta$ and antipode $\Sm^{-1}$. The $\Psi, f$ restrict naturally to isomorphisms of $U_q(\Glie)$ and $U_q(\Glie')$, still denoted by $\Psi,f$. Let $f(z) \in 1+z\BC[[z]]$ and $g(z) \in 1+z^{-1}\BC[[z^{-1}]]$. The following are morphisms of superalgebras:
\begin{eqnarray}
&& \ev_a: U_q(\Gaff) \longrightarrow U_q(\Glie), \quad s_{ij}(z) \mapsto s_{ij} - za t_{ij},\quad t_{ij}(z) \mapsto t_{ij} - z^{-1}a^{-1}s_{ij},  \label{equ: evaluation morphism}  \\
&& \phi_{[f(z),g(z)]}: U_q(\Gaff) \longrightarrow U_q(\Gaff),\quad s_{ij}(z) \mapsto f(z)s_{ij}(z),\quad t_{ij}(z) \mapsto g(z) t_{ij}(z).  \label{equ: automorphism power series}
\end{eqnarray}
These morphisms satisfy natural compatibility relations. For example,
\begin{displaymath}
\Psi \circ \ev_a = \ev_{a^{-1}} \circ \Psi: U_q(\Gaff) \longrightarrow U_q(\Glie), \quad f \circ \ev_a' = \ev_a \circ f: U_q(\Gafft) \longrightarrow U_q(\Glie). 
\end{displaymath}
\subsection{Highest $\ell$-weight modules.}\label{subsec: classification} The Hopf superalgebra $U_q(\Gaff)$ is $\BP$-graded: an element $x \in U_q(\Gaff)$ is of weight $\lambda \in \BP$ if $s_{ii}^{(0)} x t_{ii}^{(0)} = q^{(\lambda,\epsilon_i)} x$ for all $i \in I$. Indeed, $s_{ij}^{(n)}$ and $t_{ij}^{(n)}$ are of weight $\epsilon_i-\epsilon_j$. Such a $\BP$-grading descends to $U_q(\Glie)$. For a $U_q(\Glie)$-module $V$ and $\lambda \in \BP$, we set $(V)_{\lambda}$ to be the subspace of $V$ formed of vectors $v$ such that $s_{ii}^{(0)} v = q^{(\lambda,\epsilon_i)} v$ for all $i \in I$, and call it the weight space of weight $\lambda$. 

Let $V$ be a $U_q(\Gaff)$-module. A non-zero vector $v \in V$ is called a {\it highest $\ell$-weight vector} if it is a common eigenvector for the $s_{ii}^{(n)},t_{ii}^{(n)}$ and it is annihilated by the $s_{ij}^{(n)},t_{ij}^{(n)}$ with $i < j$. $V$ is called a highest $\ell$-weight module if it is generated as a $U_q(\Gaff)$-module by a highest $\ell$-weight vector. Similarly, there are the notions of {\it lowest $\ell$-weight vector/module} by replacing $(i<j)$ with $(i>j)$.  By dropping the $(n)$, we obtain the notions of {\it highest/lowest weight vector/module} related to $U_q(\Glie)$-modules. According to Equation \eqref{for: coproduct for quantum affine superalgebra S}, a tensor product of highest/lowest ($\ell$-)weight vectors is again a highest/lowest ($\ell$-)weight vector. This is not necessarily true when replacing \lq\lq vector\rq\rq\ with \lq\lq module\rq\rq, yet we have the following:
\begin{lem} \cite[Lemma4.5]{Z2}   \label{lem: cyclicity Chari}
Let $V_+$ (resp. $V_-$) be a $U_q(\Gaff)$-module of highest (resp. lowest) $\ell$-weight. Let $v_+ \in V_+$ (resp. $v_- \in V_-$) be a highest (resp. lowest) $\ell$-weight vector. Then the $U_q(\Gaff)$-module $V_+ \otimes V_-$ (resp. $V_- \otimes V_+$) is generated by $v_+ \otimes v_-$ (resp. $v_- \otimes v_+$).
\end{lem}  
The proof of this lemma in \cite{Z2} utilized the Drinfeld new realization of $U_q(\Gaff)$.

For $\lambda = \sum_i \lambda_i\epsilon_i \in \BP$, let $L(\lambda)$ be the simple $U_q(\Glie)$-module of highest weight $\lambda$; it is finite-dimensional if and only if $\lambda_i \geq \lambda_{i+1}$ for $i \neq M$; see \cite{Zr}.
\begin{example}   \label{example: natural representation quantum superalgebra}
 The vector representation $\rho_{(1)}$ of $U_q(\Glie)$ on $\BV$ is defined by
\begin{align*}
& \rho_{(1)} (s_{ii}) = q_i  E_{ii} + \sum_{j\neq i} E_{jj} = \rho_{(1)}(t_{ii}^{-1}) \quad \textrm{for}\ i \in I,  \\
& \rho_{(1)} (s_{ij}) = (q_i - q_i^{-1}) E_{ij},\quad \rho_{(1)}(t_{ji}) = - (q_i - q_i^{-1}) E_{ji} \quad \textrm{for}\ i < j.
\end{align*}
$v_1$ (resp. $v_{\kappa}$) is a highest (resp. lowest) weight vector and $v_i$ is of weight $\epsilon_i$ for $i \in I$. The resulting $U_q(\Glie)$-module $\BV$ is $L(\epsilon_1)$. For $a \in \BC^{\times}$, let $\BV(a)$ denote the $U_q(\Gaff)$-module $\ev_a^* \BV$.
\end{example}
\begin{example}  \label{example: negative natural representation}
The pull back $\Psi^*\BV$ of the $U_q(\Glie)$-module $\BV$ by the isomorphism $\Psi$ in \eqref{equ: transposition in quantum affine superalgebra} defines another representation $\rho_{(1)} \Psi$ of $U_q(\Glie)$ on $\BV$. Let $\BW$ be the corresponding $U_q(\Glie)$-module. For distinction, let us write $w_i := \Psi^*v_i$ for $i \in I$. Now $w_{\kappa}$ (resp. $w_1$) is a highest (resp. lowest) weight vector and $w_i$ is of weight $-\epsilon_i$. So $\BW = L(-\epsilon_{\kappa})$ as $U_q(\Glie)$-modules. For $a \in \BC^{\times}$, let $\BW(a)$ denote the $U_q(\Gaff)$-module $\ev_a^* \BW$.
\end{example}

To motivate the definition fundamental modules, let us recall the highest $\ell$-weight classification of finite-dimensional simple $U_q(\Gaff)$-modules from \cite{Z1,Z3}. Let $S$ be such a module. Firstly $S$ contains a unique (up to scalar multiple) highest $\ell$-weight vector $v$. Secondly, for $1\leq i < \kappa$, the eigenvalues of $s_{ii}(z)s_{i+1,i+1}(z)^{-1}, t_{ii}(z) t_{i+1,i+1}(z)^{-1}$ associated to $v$ turn out to be the $z = 0, z = \infty$ Taylor expansions of a rational function $f_i(z) \in \BC(z)$ satisfying: \footnote{We use rational functions instead of Drinfeld polynomials in \cite[Prop.4.12]{Z1} and \cite[Prop.6.7]{Z3}; they are indeed equivalent under the Ding--Frenkel homomorphism reviewed in \cite[Theorem 3.12]{Z2}.}
\begin{itemize}
\item[(1)] if $i \neq M$, then $f_i(z)$ is a product of the $q_i \frac{1-za}{1-zaq_i^2}$ with $a \in \BC^{\times}$;
\item[(2)] $f_M(z)$ is a product of the $c \frac{1-za}{1-zac^2}$ with $c,a \in \BC^{\times}$. 
\end{itemize}
Thirdly $S \mapsto \Pi(S) := (f_i(z))_{1\leq i < \kappa} \in \BC(z)^{\kappa-1}$ establishes a bijection between the isomorphism classes of finite-dimensional simple $U_q(\Gaff)$-modules up to tensor products with one-dimensional modules and elements in $\BC(z)^{\kappa-1}$ with conditions (1)--(2). Lastly, if $S,S'$ are two finite-dimensional simple $U_q(\Gaff)$-modules, then by Equation \eqref{for: coproduct for quantum affine superalgebra S}, $S \otimes S'$ contains a highest $\ell$-weight vector which gives rise to another simple module $S''$ with
$$ \Pi(S'') = \Pi(S)\Pi(S') \in \BC(z)^{\kappa-1}.$$

\begin{rem} \label{rem: highest l weight of evaluation modules}
Let $\lambda = \sum_{i} \lambda_i \epsilon_i \in \BP$ and $a \in \BC^{\times}$. When $L(\lambda)$ is finite-dimensional,
\begin{align*}
\Pi(\ev_a^*L(\lambda)) &= ( \frac{q^{\lambda_1} - za q^{-\lambda_1}}{q^{\lambda_2} - za q^{-\lambda_2}}, \cdots, \frac{q^{\lambda_{M-1}} - za q^{-\lambda_{M-1}}}{q^{\lambda_M}-zaq^{-\lambda_M}},  \frac{q^{\lambda_{M}} - za q^{-\lambda_{M}}}{q^{-\lambda_{M+1}}-zaq^{\lambda_{M+1}}}, \\
 &\quad \ \frac{q^{-\lambda_{M+1}} - za q^{\lambda_{M+1}}}{q^{-\lambda_{M+2}}-zaq^{\lambda_{M+2}}},\cdots, \frac{q^{-\lambda_{\kappa-1}} - za q^{\lambda_{\kappa-1}}}{q^{-\lambda_{\kappa}}-zaq^{\lambda_{\kappa}}} ).
\end{align*} 
\end{rem}

\subsection{Fundamental representations.} We are interested in such simple $U_q(\Gaff)$-modules $S$ that all but one components of $\Pi(S)$ are $1$. They can be constructed by fusion procedures.
\begin{defi}   \label{def: fundamental representations}
Let $a \in \BC^{\times}$ and $s,t \in \BZ_{>0}$ be such that $s \leq M$ and $t \leq N$. The sub-$U_q(\Gaff)$-module of $\otimes_{j=1}^s \BV(aq^{-2j})$ generated by $v_{\kappa}^{\otimes s}$ is called a positive fundamental module and denoted by $V_{s,a}^+$. The sub-$U_q(\Gaff)$-module of $\otimes_{j=1}^t \BW(aq^{2j})$ generated by $w_1^{\otimes t}$ is called a negative fundamental module and denoted by $V_{t,a}^-$.
\end{defi}
The terminology \lq\lq positive/negative\rq\rq\ will be justified at the end of this section. The following proposition will be proved in \S \ref{sec: simplicity} (page~\pageref{page:simplicity}) when twisted duals are introduced.
\begin{prop} \label{prop: simplicity of fundamental modules}
$V_{s,a}^+$ and $V_{t,a}^-$ are simple $U_q(\Gaff)$-modules for $1\leq s \leq M, 1\leq t \leq N$.
\end{prop}
 The following theorem is a special case of more general results in \cite{BKK}. 
\begin{theorem}  \label{thm: BKK Schur-Weyl duality}   
Let $1\leq s \leq M$. The $U_q(\Glie)$-module $L(\epsilon_1)^{\otimes s} = \BV^{\otimes s}$ is completely reducible. Its submodule generated by $v_{\kappa}^{\otimes s}$ is isomorphic to $L(\epsilon_1+\epsilon_2+\cdots+\epsilon_s)$ whose weight spaces are one-dimensional and whose weights are the  $\epsilon_{i_1} + \epsilon_{i_2} + \cdots + \epsilon_{i_s}$ where: $i_1 \leq i_2 \leq \cdots \leq i_s$; if $i_k = i_{k+1}$ then $i_k > M$. Denote this sub-$U_q(\Glie)$-module by $V_s^+$.
\end{theorem}
In \cite{BKK}, to certain $\lambda \in \BP$ is associated an $(M,N)$-hook Young diagram $Y^{\lambda}$ (an ordinary Young diagram without box at the $(M+1,N+1)$-position). Such $L(\lambda)$ has a crystal basis in the sense of Kashiwara labeled by semi-standard tableaux (assignment of numbers between $1$ and $\kappa$ to the boxes according to certain rules) in $Y^{\lambda}$. In the above theorem, $\epsilon_1+\epsilon_2+\cdots +\epsilon_s$ corresponds to the Young diagram with $s$ boxes in one column. The conditions of the $i_k$ are exactly those of being a semi-standard tableau. For example, when $M=N=2$, the weights (and crystal basis vectors) are indexed by the following tableaux: 
$$M=N=2, V_2^+ = L(\epsilon_1+\epsilon_2):\quad \ \begin{Young}
1 \cr
2 \cr
\end{Young},\quad  \begin{Young}
1 \cr
3 \cr
\end{Young},\quad \begin{Young}
1 \cr
4 \cr
\end{Young},\quad  \begin{Young}
2 \cr
3 \cr
\end{Young},\quad \begin{Young}
2 \cr
4 \cr
\end{Young},\quad \begin{Young}
3 \cr
3 \cr
\end{Young}, \quad \begin{Young}
3 \cr
4 \cr
\end{Young},\quad \begin{Young}
4 \cr
4 \cr
\end{Young}. $$

\begin{lem}     \label{lem: highest weight vector even fund}
For $1\leq s \leq M$, the sub-$U_q(\Glie)$-module $V_s^+$ has a highest weight vector
\begin{displaymath}
v^{(s)} := \sum_{\sigma \in \mathfrak{S}_s} (-q)^{l(\sigma)} v_{\sigma(1)} \otimes v_{\sigma(2)} \otimes \cdots \otimes v_{\sigma(s)} \in \BV^{\otimes s}.
\end{displaymath}
\end{lem}
Here $l(\sigma)$ denotes the length of a permutation $\sigma \in \mathfrak{S}_s$.
\begin{proof}
Let us first prove that $v^{(s)}$ is a highest weight vector. By the weight grading on $\BV^{\otimes s}$ and $U_q(\Glie)$, it is enough to show that $s_{jk} v^{(s)} = 0$ for $1\leq j < k \leq s$. By using the relations of the $s_{jk}$ (see the proof of \cite[Prop.4.6]{Z3}), we can assume that $k=j+1$.   Let $X$ be the set of permutations $\sigma$ such that $\sigma^{-1}(j) < \sigma^{-1}(k)$. Let $\theta$ be the simple transposition $(j,k)$. Then $\mathfrak{S}_s$ is a disjoint union of $X$ and $\theta X$, and $l(\theta \sigma) = l(\sigma) + 1$ whenever $\sigma \in X$. By using the formulas in Example \ref{example: natural representation quantum superalgebra}, we are reduced to the case $s = 2=k$ and $j=1$. Now $v_1\otimes v_2 - q v_2\otimes v_1$ is easily shown to be of highest weight.

$v^{(s)}$ generates a sub-$U_q(\Glie)$-module $S$. It is of highest weight $\epsilon_1+\epsilon_2+\cdots +\epsilon_s$ and completely reducible by Theorem \ref{thm: BKK Schur-Weyl duality}; $S$ must be simple, and $(S)_{s\epsilon_{\kappa}} = (\BV^{\otimes s})_{s\epsilon_{\kappa}} = \BC v_{\kappa}^{\otimes s}$. This implies that $V_s^+ \subseteq S$. Since $S$ is a simple $U_q(\Glie)$-module, $V_s^+ = S$ and $v^{(s)} \in V_s^+$. 
\end{proof}
\begin{lem}  \label{lem: fundamental modules}                           
Let $a \in \BC^{\times}$ and $1\leq s \leq M$. as $U_q(\Gaff)$-modules, $V_{s,a}^+ \cong \ev_{aq^{-2s}}^* (V_s^+) \otimes D$ for some one-dimensional module $D$. As sub-$U_q(\Glie)$-modules of $\BV^{\otimes s}$, we have $V_{s,a}^+ = V_s^+$. 
\end{lem}
\begin{proof}
It is enough to prove the first part, as $V_s^+ \subseteq V_{s,a}^+$. Let $f_i(z)$ be the eigenvalue of $s_{ii}(z)$ associated to the lowest $\ell$-weight vector $v_{\kappa}^{\otimes s}$ in $\otimes_{j=1}^s \BV(aq^{-2j})$. Then $f_i(z) = \prod_{j=1}^s (1-zaq^{-2j})$ for $i < \kappa$ and $f_{\kappa}(z) = \prod_{j=1}^s (q^{-1}-zaq^{-2j+1})$. Similar statements hold for $t_{ii}(z)$. Now set $D = \phi_{[f(z),g(z)]}^* \BC$ where $\BC$ is the one-dimensional trivial $U_q(\Gaff)$-module and $f(z) = \prod_{j=1}^{s-1}(1-zaq^{-2j}), g(z) = \prod_{j=1}^{s-1}(1-z^{-1}a^{-1}q^{2j})$. Then the lowest $\ell$-weight vector of $\ev_{aq^{-2s}}^*(V_s^+) \otimes D$ and $v_{\kappa}^{\otimes s}$ have the same eigenvalues of the $s_{ii}(z),t_{ii}(z)$. Since $V_{s,a}^+$ is a simple $U_q(\Gaff)$-module, it must be isomorphic to $\ev_{aq^{-2s}}^*(V_s^+) \otimes D$.
\end{proof}
Similar results as in the above two lemmas hold true for negative fundamental modules.
\begin{lem} \label{lem: negative fundamental module}
Let $a \in \BC^{\times}$ and $1\leq t \leq N$. Let $V_{t}^-$ be the sub-$U_q(\Glie)$-module of $\BW^{\otimes t}$ generated by $w_1^{\otimes t}$. Then as sub-$U_q(\Glie)$-modules $V_{t,a}^- = V_t^-$, as $U_q(\Gaff)$-modules $V_{t,a}^- \cong  \ev_{aq^2}^* (V_t^-) \otimes D$ for some one-dimensional module $D$, and $V_t^-$ has a highest weight vector
\begin{equation*}  
w^{(t)} = \sum_{\sigma \in \mathfrak{S}_t} (-q)^{l(\sigma)} w_{\kappa-t+\sigma(1)} \otimes w_{\kappa-t+\sigma(2)} \otimes \cdots \otimes w_{\kappa-t+\sigma(t)} \in \BW^{\otimes t}.
\end{equation*}
\end{lem}
\begin{proof}
Let $v_i',\BV',\epsilon_i',\BP',V_{t,a}'^+$ be the corresponding objects for $\Glie'$ (so $1\leq t \leq N$). By comparing the highest $\ell$-weight vectors we get a $U_q(\Gafft)$-linear isomorphism $\theta: f^* \BW(a) \cong \BV'(a) \otimes \BC_{\odd}, w_i \mapsto x_i v_{\widehat{i}}' \otimes  [\odd]$ where $x_i \in \BC^{\times}$ for $i \in I$ and $\BC_{\odd} = \BC [\odd]$ is the one-dimensional odd module over $U_q(\Gafft)$. The graded permutation $\BV'(a) \otimes \BC_{\odd} \longrightarrow \BC_{\odd} \otimes \BV'(a)$, being $U_q(\Gafft)$-linear, induces an isomorphism of $U_q(\Gafft)$-modules 
$$\Sigma_t: \otimes_{j=t}^1 (\BV'(aq^{2j}) \otimes \BC_{\odd})\longrightarrow(\otimes_{j=t}^1 \BV'(aq^{2j})) \otimes \BC_{\odd}^{\otimes t}.$$ 
The assignment $\otimes_{j=1}^tu_j \mapsto (-1)^{\sum_{k<l}|u_k||u_l|} \otimes_{j=t}^1 u_j$ extends to a $U_q(\Gafft)$-linear isomorphism $\sigma_t: f^*(\otimes_{j=1}^t\BW(aq^{2j})) \longrightarrow \otimes_{j=t}^1 f^*\BW(aq^{2j})$ by Equation \eqref{equ: from gl(M,N) to gl(N,M)}. The composition 
\begin{displaymath}
 f^*(\otimes_{j=1}^t\BW(aq^{2j})) \xrightarrow{\sigma_t} \otimes_{j=t}^1 f^*\BW(aq^{2j}) \xrightarrow{\theta^{\otimes t}} \otimes_{j=t}^1 (\BV'(aq^{2j}) \otimes \BC_{\odd}) \xrightarrow{\Sigma_t} (\otimes_{j=t}^1 \BV'(aq^{2j})) \otimes \BC_{\odd}^{\otimes t}
\end{displaymath}
restricts to a $U_q(\Gafft)$-module isomorphism $\vartheta_t: f^*V_{t,a}^- \longrightarrow V_{t,aq^{2t+2}}'^+ \otimes \BC_{\odd}^{\otimes t}$. Lemmas \ref{lem: highest weight vector even fund}--\ref{lem: fundamental modules} for the $U_q(\Gafft)$-module $V_{t,a}'^+$ can be translated into those for the $U_q(\Gaff)$-module $V_{t,a}^-$ via $\vartheta_t$.
\end{proof}

By comparing the weights, we see that as $U_q(\Glie)$-modules, $V_t^- \cong L(t \epsilon_1)^*$. Here the dual space $V^* = \hom(V,\BC)$ of a $U_q(\Glie)$-module $V$ is endowed with the $U_q(\Glie)$-module structure:
\begin{displaymath}
\langle x l, v \rangle := (-1)^{|l||x|} \langle l, \Sm (x) v \rangle \quad \mathrm{for}\ x \in U_q(\Glie), l \in V^*, v \in V.  
\end{displaymath} 
In \cite{BKK}, $t\epsilon_1$ corresponds to the Young diagram with $t$ boxes in one row. Again consider the example $M=N=2=t$.
\begin{align*}
&M=N=2, V_2^- = L(2\epsilon_1)^*: \\
&\begin{Young}
1 & 1 \cr
\end{Young}^*, \quad \begin{Young}
1 & 2 \cr
\end{Young}^*, \quad \begin{Young}
1 & 3 \cr
\end{Young}^*, \quad \begin{Young}
1 & 4 \cr
\end{Young}^*, \quad \begin{Young}
2 & 2 \cr
\end{Young}^*, \quad \begin{Young}
2 & 3 \cr
\end{Young}^*, \quad \begin{Young}
2 & 4 \cr
\end{Young}^*, \quad \begin{Young}
3 & 4 \cr
\end{Young}^*.
\end{align*} 

\begin{example}   \label{example: tensor square}
$V_2^+$ is spanned by the $v_l^{\otimes 2}$ and $v_i \otimes v_j - (-1)^{|i||j|} q v_j \otimes v_i$ with $M< l\leq \kappa$ and $1\leq i < j \leq \kappa$; see \cite[\S 2]{Z2}. By Lemma \ref{lem: negative fundamental module}, $V_2^-$ is spanned by the $w_l^{\otimes 2}$ and $w_i \otimes w_j + (-1)^{|i||j|} q w_j \otimes w_i$ with $1\leq l\leq M$ and $1\leq i < j \leq \kappa$. In general, $V_s^+$ (resp. $V_t^-$) is seen as quantum exterior power $\bigwedge_q^s \BV$ (resp. symmetric power $S_q^t \BW$).
\end{example}
From Lemmas \ref{lem: fundamental modules}--\ref{lem: negative fundamental module} and Remark \ref{rem: highest l weight of evaluation modules}, we see that for $1\leq s \leq M$ and $1\leq t \leq N$,
$$\Pi(V_{s,a}^+) = (1^{s-1},q\frac{1-zaq^{-2s-2}}{1-zaq^{-2s}},1^{\kappa-1-s}),\quad \Pi(V_{t,a}^-) = (1^{\kappa-1-t},q^{-1}\frac{1-zaq^2}{1-za},1^{t-1}).$$
$\pm$ indicate the positive/negative powers of $q$ in $\Pi(V^{\pm}_{s,a})|_{z=0}$. Our definition of fundamental modules, viewed in terms of the highest $\ell$-weight classification, is then in accordance with that in the non-graded case \cite[Definition 3.4]{CH}; see also Footnote~\ref{footnote:Okado}. 
\section{$R$--matrices of fundamental representations}   \label{sec: spin r matrix}
The aim of this section is to construct $U_q(\Gaff)$-linear maps between fundamental modules. The following lemma is our starting point. Its proof, postponed to \S \ref{sec: simplicity} (page~\pageref{page:proof r-matrix}), is independent of denominators of normalized $R$-matrices; see the remark before Theorem \ref{thm: simplicity of tensor products}.
\begin{lem}   \label{lem: tensor product of odd even}
Let $V, W$ be two fundamental modules with highest $\ell$-weight vectors $v,w$ respectively. For $a,b \in \BC^{\times}$, denote $V_a := \Phi_a^* V$ and $W_b := \Phi_b^*W$. There exists a finite set $\Sigma \subset \BC^{\times}$ such that: if $\frac{a}{b} \notin \Sigma$, then $V_a \otimes W_b$ is a simple $U_q(\Gaff)$-module and there exists a unique $U_q(\Gaff)$-module isomorphism $V_a \otimes W_b \longrightarrow W_b \otimes V_a$ sending $v \otimes w$ to $w \otimes v$.
\end{lem}  
\begin{rem}
Presumably the quantum affine superalgebra $U_q(\Gaff)$ admits a universal $R$-matrix $\mathcal{R}(z) \in U_q^+(\Gaff) \widehat{\otimes} U_q^-(\Gaff)[[z]]$. Here $U_q^{+}(\Gaff)$ (resp. $U_q^-(\Gaff)$) is the subalgebra generated by the $(s_{ii}^{(0)})^{\pm 1}, s_{ij}^{(n)}$ (resp. the $(t_{ii}^{(0)})^{-1}, t_{ij}^{(n)}$), and $\widehat{\otimes}$ is a completed tensor product arising from the weight grading. The $U_q(\Gaff)$-module isomorphism in the above lemma can then be thought of as (up to a scalar factor which is a meromorphic function in $\frac{a}{b}$) the specialization $c_{V,W} \mathcal{R}(\frac{a}{b})|_{V,W}$. In \cite[\S 3.3.6]{Z2} a Hopf pairing $\widehat{\varphi}$ between these two subalgebras was constructed. The author believes that $\widehat{\varphi}$ is non-degenerate and its Casimir element gives $\mathcal{R}(z)$. See \cite{Z4} for a proof in the case $\mathfrak{gl}(1,1)$.
\end{rem}

The following result is taken from \cite[Lemma 4.6]{Z2}.
\begin{lem}    \label{lem: ++ or -- matrix}
For $a, b \in \BC^{\times}$, $c_{\BV,\BV} R(a,b): \BV(a) \otimes \BV(b) \longrightarrow \BV(b) \otimes \BV(a)$ is $U_q(\Gaff)$-linear. 
\end{lem}

Let $F: \BV(a) \otimes \BW(b) \longrightarrow \BW(b) \otimes \BV(a)$ be a $U_q(\Gaff)$-linear map sending $v_1 \otimes w_{\kappa}$ to $w_{\kappa} \otimes v_1$; by Lemma \ref{lem: tensor product of odd even}, $F$ exists when $\frac{a}{b}$ is generic. We shall compute the $F(v_i \otimes w_j)$. 

\noindent \underline{Step I.} For $k \neq l$, $\BC v_k \otimes w_l$ is the weight space of $\BV(a)\otimes \BW(b)$ of weight $\epsilon_k - \epsilon_l$. The zero weight space is spanned by the $v_i \otimes w_i$. Similar statements hold for $\BW(b) \otimes \BV(a)$. Since $F$ respects the weight spaces, there exist $\lambda_{ij},\theta_{kl}$ for $i,j,k,l \in I$ and $k \neq l$ such that $\theta_{1\kappa } = 1$,
\begin{displaymath}
F(v_j \otimes w_j) = \sum_{i\in I} \lambda_{ij} w_i \otimes v_i, \quad F(v_k \otimes w_l) = \theta_{kl} w_l \otimes v_k \quad \mathrm{for}\ k \neq l.
\end{displaymath}

\noindent \underline{Step II.} Let $i,j,k \in I$ be such that $i < j < k$. Compare $\theta_{ij}$ with $\theta_{ik}$. We have
\begin{displaymath}
(2.a)\quad F(s_{jk} (v_i \otimes w_j)) = s_{jk} F(v_i \otimes w_j) = \theta_{ij} s_{jk} (w_j \otimes v_i).
\end{displaymath}
By Equation \eqref{for: coproduct for quantum affine superalgebra S} and Examples \ref{example: natural representation quantum superalgebra}--\ref{example: negative natural representation}:
\begin{align*}
s_{jk}(v_i\otimes w_j) &= \sum_{l=j}^k (s_{jl} \otimes s_{lk})(v_i \otimes w_j) = (s_{jj} \otimes s_{jk}) (v_i \otimes w_j) = (-1)^{|i|(|j|+|k|)} s_{jj}v_i \otimes s_{jk}w_j  \\
&= v_i \otimes s_{jk}w_j = (q_k^{-1}-q_k) v_i \otimes w_k,  \\
s_{jk}(w_j\otimes v_i) &= \sum_{l=j}^k (s_{jl} \otimes s_{lk})(w_j \otimes v_i) = (s_{jk}\otimes s_{kk})(w_j \otimes v_i) = (q_k^{-1}-q_k)w_k \otimes v_i.
\end{align*}
It follows from (2.a) that $\theta_{ij} = \theta_{ik}$. Next compare $\theta_{ik}$ and $\theta_{jk}$ by using
\begin{align*}
t_{ji}(v_i\otimes w_k) &= (t_{ji}\otimes t_{ii})(v_i\otimes w_k) = (q_i^{-1}-q_i)v_j\otimes w_k, \\
t_{ji}(w_k\otimes v_i) &= (t_{jj}\otimes t_{ji})(w_k\otimes v_i) = (-1)^{|i|+|j|} (q_i^{-1}-q_i) w_k\otimes v_j.
\end{align*}
Applying $Ft_{ji}$ and $t_{ji} F$ to $v_i\otimes w_k$ we get $\theta_{ik} = (-1)^{|i|+|j|}\theta_{jk}$. Now compute
\begin{align*}
t_{kj}(v_j\otimes w_i) &= (t_{kj}\otimes t_{jj})(v_j\otimes w_i) = (q_j^{-1}-q_j)v_k \otimes w_i, \\
t_{kj}(w_i\otimes v_j) &= (t_{kk}\otimes t_{kj})(w_i\otimes v_j) = (q_j^{-1}-q_j)w_i\otimes v_k.
\end{align*}
Applying $F t_{kj}$ and $t_{kj}F$ to $v_j\otimes w_i$ we get $\theta_{ji} = \theta_{ki}$. At last consider
\begin{align*}
s_{ij}(v_k\otimes w_i) &= (s_{ii}\otimes s_{ij})(v_k\otimes w_i) = (-1)^{|i|+|j|}(q_j^{-1}-q_j) v_k \otimes w_j, \\
s_{ij}(w_i\otimes v_k) &= (s_{ij}\otimes s_{jj})(w_i\otimes v_k) = (q_j^{-1}-q_j) w_j \otimes v_k.
\end{align*}
Applying $Fs_{ij}$ and $s_{ij}F$ to $v_k\otimes w_i$ we get $\theta_{ki} = (-1)^{|i|+|j|}\theta_{kj}$.

\noindent \underline{Step III.} We assume that $j < k$. Let us compare $\theta_{jk}$ and $\theta_{kj}$. Compute
\begin{align*}
s_{jk}(v_k \otimes w_j) &= (s_{jj}\otimes s_{jk} + s_{jk}\otimes s_{kk})(v_k\otimes w_j) = (q_j-q_j^{-1})(v_j\otimes w_j-v_k\otimes w_k), \\
t_{kj}(v_j \otimes w_k) &= (t_{kk}\otimes t_{kj} + t_{kj}\otimes t_{jj})(v_j \otimes w_k) = (q_j-q_j^{-1})(v_j \otimes w_j - v_k \otimes w_k).
\end{align*}
By applying $F$ to the above identities, we get $\theta_{kj} s_{jk} (w_j \otimes v_k) = \theta_{jk} t_{kj}(w_k\otimes v_j)$. On the other hand, a straightforward calculation indicates that
\begin{align*}
s_{jk} (w_j \otimes v_k) &=  t_{kj}(w_k\otimes v_j) \\
&= (q_j-q_j^{-1})(q_j^{-1}w_j\otimes v_j + \sum_{j<l<k}(q_j^{-1}-q_j)w_l\otimes v_l - (-1)^{|j|+|k|} q_k w_k\otimes v_k), 
\end{align*}
It follows that $\theta_{jk} = \theta_{kj}$ and  $F(v_j\otimes w_j-v_k\otimes w_k) = \theta_{jk} (q_j-q_j^{-1})^{-1} s_{jk} (w_j \otimes v_k)$. In conclusion, for all $i,j \in I$, we have $\theta_{ij} = \theta_{1\kappa}(-1)^{|i||j|} = (-1)^{|i||j|}$.

\noindent \underline{Step IV.} Assume that $j < k$. Let us apply $F s_{jk}$ and $s_{jk}F$ to $v_k \otimes w_k$. We have
\begin{align*}
s_{jk}(v_k\otimes w_k) &= (s_{jk}\otimes s_{kk})(v_k\otimes w_k) = q_k^{-1}(q_j-q_j^{-1})v_j \otimes w_k.
\end{align*}
So $Fs_{jk}(v_k\otimes w_k) = \theta_{jk}q_k^{-1}(q_j-q_j^{-1})w_k\otimes v_j$. If $i \neq j,k$, then $s_{jk}(w_i\otimes v_i) = 0$. Otherwise,
\begin{align*}
s_{jk}(w_j\otimes v_j) &= (s_{jk}\otimes s_{kk})(w_j \otimes v_j) = (q_k^{-1}-q_k)w_k \otimes v_j, \\
s_{jk}(w_k\otimes v_k) &= (s_{jj}\otimes s_{jk})(w_k \otimes v_k) = (-1)^{|j|+|k|}(q_j-q_j^{-1}) w_k \otimes v_j.
\end{align*}
It follows that $\lambda_{kk}-\lambda_{jk} = \theta_{jk} (-1)^{|j|+|k|}q_k^{-1}$.

\noindent \underline{Step V.} Let us apply $F s_{\kappa 1}(z)$ and $s_{\kappa 1}(z)F$ to $v_1 \otimes w_{\kappa}$ by developing
\begin{align*}
s_{\kappa 1}(z) (v_1\otimes w_{\kappa}) &= (s_{{\kappa}1}(z) \otimes s_{11}(z) + s_{\kappa \kappa}(z) \otimes s_{\kappa 1}(z))(v_1 \otimes w_{\kappa}) \\
&= (q-q^{-1})za (1-zb) v_{\kappa}\otimes w_{\kappa} + (1-za)(q^{-1}-q)zb v_1 \otimes w_1, \\
s_{\kappa 1}(z)(w_{\kappa} \otimes v_1) &= \sum_l (s_{\kappa l}(z)\otimes s_{l1}(z))(w_{\kappa} \otimes v_1) \\
&= (q^{-1}-q)zb(q-zaq^{-1})w_1\otimes v_1 - (q-zbq^{-1})(q-q^{-1})za w_{\kappa} \otimes v_{\kappa} \\
&\quad + \sum_{1<l<\kappa} (-1)^{|l|} (q_l^{-1}-q_l) zb (q-q^{-1})za w_l\otimes v_l. 
\end{align*}
From the identity $Fs_{\kappa 1}(z)(v_1\otimes w_{\kappa}) = s_{\kappa 1}(z)(w_{\kappa}\otimes v_1)$ we deduce that:
\begin{align*}
F(v_1\otimes w_1) &= \frac{bq-aq^{-1}}{b-a}w_1\otimes v_1 + \frac{a(q-q^{-1})}{b-a}\sum_{l>1} w_l\otimes v_l, \\
F(v_{\kappa}\otimes w_{\kappa}) & = \frac{b(q-q^{-1})}{b-a}\sum_{l<\kappa} w_l \otimes v_l - \frac{bq^{-1}-aq}{b-a} w_{\kappa} \otimes v_{\kappa}.
\end{align*}
By using the identity for $F(v_i \otimes w_i - v_{\kappa} \otimes w_{\kappa})$ in Step III, we obtain that for $i \in I$,
\begin{align*}
&F(v_i\otimes w_i) = \sum_{l<i}\frac{b(q-q^{-1})}{b-a}w_l\otimes v_l + (-1)^{|i|}\frac{bq_i-aq_i^{-1}}{b-a}w_i\otimes v_i + \sum_{l>i}\frac{a(q-q^{-1})}{b-a}w_l\otimes v_l.
\end{align*}

Now let us introduce the matrix $R^{+-}_{a,b} := c_{\BW,\BV} F \in \End(\BV\otimes \BW)$:
\begin{eqnarray}  \label{for: positive-negative R matrix}
\begin{array}{rcl}
R_{a,b}^{+-} &=& \sum\limits_{i\neq j} E_{ii} \otimes E_{jj} + \sum\limits_{i\in I}\frac{bq_i - aq_i^{-1}}{b-a} E_{ii} \otimes E_{ii}   \\
&\ & +  \sum\limits_{i<j} \frac{a(q_j-q_j^{-1})}{b-a} E_{ji} \otimes E_{ji} +  \sum\limits_{i>j}\frac{b(q_i-q_i^{-1})}{b-a}  E_{ji} \otimes E_{ji}.
\end{array}  
\end{eqnarray}
Here by abuse of language $E_{ij}$ is also in $\End(\BW)$ sending $w_k$ to $\delta_{jk} w_i$.
\begin{lem} \label{lem: +- matrix}
$c_{\BV,\BW}R_{a,b}^{+-}: \BV(a) \otimes \BW(b) \longrightarrow \BW(b) \otimes \BV(a)$ is a morphism of $U_q(\Gaff)$-modules provided that $a,b \in \BC^{\times}$ and $a \neq b$. 
\end{lem}
\begin{proof}
Let $\pi_1,\pi_2$ denote the representations of $U_q(\Gaff)$ on $\BV(a)\otimes \BW(b)$ and $\BW(b)\otimes \BV(a)$ respectively. We need to show that for $x$ an arbitrary RTT generator of $U_q(\Gaff)$,
$$(*)_x:\quad (a-b)c_{\BV,\BW}R_{a,b}^{+-} \pi_1(x) = (a-b)\pi_2(x) c_{\BV,\BW}R_{a,b}^{+-}. $$
By Examples \ref{example: natural representation quantum superalgebra}--\ref{example: negative natural representation} and Equation \eqref{for: coproduct for quantum affine superalgebra S}, $\pi_1(x), \pi_2(x)$ are polynomials in $a,b$. Combining Equation \eqref{for: positive-negative R matrix}, we see that $(*)_x$ is a polynomial equation in $a,b$. Lemma \ref{lem: ++ or -- matrix} and the above explicit computation of $F$ prove $(*)_x$ when $\frac{a}{b}$ is in the complementary of a finite subset of $\BC^{\times}$.  By polynomiality $(*)_x$ is true for all $a,b\in \BC^{\times}$.
\end{proof}

Let us define two classes of fusion $R$-matrices: $s,t \in \BZ_{>0}$, 
\begin{align}
R^{s|t}_{a,b} &:= \prod_{j=t}^1 \prod_{i=1}^s (R^{+-}_{aq^{-2i},bq^{2j}})_{i,s+j} \in \End(\BV^{\otimes s} \otimes \BW^{\otimes t}),  \label{equ: fusion even odd}  \\
R^{s,t}_{a,b} &:= \prod_{j=t}^1 \prod_{i=1}^s R(aq^{-2i},bq^{-2j})_{i,s+j} = \prod_{i=1}^s \prod_{j=t}^1 R(aq^{-2i},bq^{-2j})_{i,s+j}  \in \End(\BV^{\otimes s+t}).  \label{equ: fusion even even}
\end{align}
The last equation holds by definition of the tensor subscripts in \S \ref{subsec: quantum superalgebra}.
\begin{lem}  \label{lem: normalized r matrix even odd}
Let $1 \leq s \leq M$ and $t \in \BZ_{>0}$. 
\begin{itemize}
\item[(A)] Suppose $t \leq M$. Then for all $a,b \in \BC^{\times}$, the linear map 
\begin{displaymath}
c_{\BV^{\otimes s}, \BV^{\otimes t}} R^{s,t}_{a,b}: (\otimes_{i=1}^s \BV(aq^{-2i})) \otimes (\otimes_{j=1}^t \BV(bq^{-2j})) \longrightarrow (\otimes_{j=1}^t \BV(bq^{-2j})) \otimes (\otimes_{i=1}^s \BV(aq^{-2i}))
\end{displaymath}
restricts to a $U_q(\Gaff)$-module map $V_{s,a}^+ \otimes V_{t,b}^+ \longrightarrow V_{t,b}^+ \otimes V_{s,a}^+$. In particular, there exist $X, Y \in \BC[a,b]$ such that $R_{a,b}^{s,t}(v^{(s)} \otimes v^{(t)}) = X v^{(s)} \otimes v^{(t)}, R_{a,b}^{s,t}(v_{\kappa}^{\otimes s+t}) = Y v_{\kappa}^{\otimes s+t}$.
\item[(B)] Suppose $t \leq N$. Let $a,b \in \BC^{\times}$ be such that $aq^{-2i} \neq bq^{2j}$ for all $1 \leq i \leq s$ and $1 \leq j \leq t$. Then the linear map
\begin{displaymath}
c_{\BV^{\otimes s}, \BW^{\otimes t}} R^{s|t}_{a,b}: (\otimes_{i=1}^s \BV(aq^{-2i})) \otimes (\otimes_{j=1}^t \BW(bq^{2j})) \longrightarrow (\otimes_{j=1}^t \BW(bq^{2j})) \otimes (\otimes_{i=1}^s \BV(aq^{-2i}))
\end{displaymath}
restricts to a $U_q(\Gaff)$-module map $V_{s,a}^+ \otimes V_{t,b}^- \longrightarrow V_{t,b}^- \otimes V_{s,a}^+$. Furthermore we have $R_{a,b}^{s|t}(v^{(s)} \otimes w^{(t)}) = v^{(s)} \otimes w^{(t)}, R_{a,b}^{s|t}(v_{\kappa}^{\otimes s} \otimes w_1^{\otimes t}) = v_{\kappa}^{\otimes s} \otimes w_1^{\otimes t}$.
\end{itemize}
\end{lem}
\begin{proof}
We shall prove (A); the same idea goes for (B). By Lemma \ref{lem: ++ or -- matrix}, $F_{a,b} := c_{\BV^{\otimes s},\BV^{\otimes t}} R^{s,t}_{a,b}$ is indeed $U_q(\Gaff)$-linear. By Equations \eqref{for: Perk-Schultz matrix coefficients} and \eqref{equ: fusion even even},
\begin{displaymath}
(1):\quad F_{a,b}(v_{\kappa}^{\otimes s+t}) = Y v_{\kappa}^{\otimes s+t},\quad Y := \prod_{i=1}^s\prod_{j=1}^t (aq^{-2i-1} - bq^{-2j+1}).
\end{displaymath}
It is therefore enough to show the following (by Lemmas \ref{lem: fundamental modules}--\ref{lem: negative fundamental module} $V_{s,a}^+ = V_s^+$)
$$ (2):\quad  F_{a,b}(V_{s}^+ \otimes V_{t}^+) \subseteq V_{t}^+ \otimes V_{s}^+, \quad F_{a,b}(v^{(s)} \otimes v^{(t)}) \in \BC v^{(t)} \otimes v^{(s)}.$$ 
By Equation \eqref{for: Perk-Schultz matrix coefficients} the matrix coefficients of $F_{a,b} \in \End (\BV^{\otimes s+t})$ are polynomial in $a,b$. Let $\Sigma \subset \BC^{\times}$ be as in Lemma \ref{lem: tensor product of odd even}, so that $S_1:= V_{s,a}^+ \otimes V_{t,b}^+ \cong V_{t,b}^+ \otimes V_{s,a}^+=:S_2$ are simple $U_q(\Gaff)$-modules whenever $\frac{a}{b} \notin \Sigma$. We show that (2) holds for $\frac{a}{b} \in \BC^{\times}\setminus \Sigma$. This will imply (2) for all $a,b \in \BC^{\times}$ by polynomiality, as in the proof of Lemma \ref{lem: +- matrix}. 

Let $\frac{a}{b} \notin \Sigma$. Then the simple $U_q(\Gaff)$-modules $S_1$ and $S_2$ are both generated by $v_{\kappa}^{\otimes s+t}$. By (1), $F_{a,b}$ restricts to $U_q(\Gaff)$-linear map $F_{a,b}: S_1\longrightarrow S_2$, and the first relation in (2) is proved. Since $v^{(s)}\otimes v^{(t)}$ and $v^{(t)}\otimes v^{(s)}$ are highest $\ell$-weight vectors of $S_1,S_2$ respectively, they must be stable by $F_{a,b}$. This proves the second relation in (2).
\end{proof}
\section{Denominators of $R$--matrices}   \label{sec: even odd}
Lemma \ref{lem: normalized r matrix even odd} together with its proof gives us three types of rational functions of $\frac{a}{b}$: the $\End(V_s^+\otimes V_t^+)$-valued $X^{-1} R_{a,b}^{s,t}|_{V_s^+\otimes V_t^+}$ and $Y^{-1} R_{a,b}^{s,t}|_{V_s^+ \otimes V_t^+}$ for $1\leq s,t\leq M$; the $\End(V_{s}^+\otimes V_t^-)$-valued $R_{a,b}^{s|t}|_{V_s^+\otimes V_t^-}$ for $1 \leq s \leq M, 1 \leq t \leq N$. The denominator of such a rational function $R(a,b)$ is defined as a homogeneous polynomial $D(a,b)$ in $a,b$ of minimal degree such that $D(a,b)R(a,b)$ is polynomial; it is well-defined up to scalar product by a non-zero complex number. In this section, we shall compute these denominators.

 In the following, if $v,w$ belong to the same vector space and $v \in \BC^{\times} w$, then we write $v \doteq w$. The denominator of the third rational function is fairly easy.

\begin{theorem}  \label{thm: pole even odd}
Let $1 \leq s \leq M$ and $1 \leq t \leq N$. The denominator of the $\End(V_{s}^+\otimes V_{t}^-)$-valued rational function $R_{a,b}^{s|t}|_{V_{s}^+\otimes V_{t}^-}$ is $bq^2-aq^{-2s}$.
\end{theorem}
\begin{proof}
By definition $v_{\kappa}^{\otimes s}$ is a lowest $\ell$-weight vector generating the simple $U_q(\Gaff)$-module $V_{s,a}^+$. Owing to Lemmas \ref{lem: cyclicity Chari} and \ref{lem: fundamental modules}, $v_{\kappa}^{\otimes s} \otimes w^{(t)}$ generates the $U_q(\Gaff)$-module $V_{s,a}^+\otimes V_{t,b}^-$. By Lemma \ref{lem: normalized r matrix even odd}, $R_{a,b}^{s|t}$ respects the $U_q(\Gaff)$-module structures. We are reduced to consider the rational function $x_{a,b} := R_{a,b}^{s|t}(v_{\kappa}^{\otimes s} \otimes w^{(t)}) \in \BV^{\otimes s} \otimes W_t^-(a,b)$.

\noindent {\it Claim.} A vector in the subspace $V_{t}^-$ of $\BW^{\otimes t}$ is uniquely determined by its components $w_i \otimes \BW^{\otimes t-1}$ with $i = \kappa$ or $i \leq \kappa - t$. \footnote{To illustrate this claim, let $A$ be the algebra generated by the $w_i$ for $i\in I$ and subject to relations $w_i w_j - (-1)^{|i||j|}q w_jw_i = w_l^2 = 0$ for $1\leq i < j \leq \kappa$ and $M<l\leq \kappa$; see Example \ref{example: tensor square}. Let $1\leq t \leq N$ and $m$ be a non-zero product of $t$ $w_i$'s. Then up to scalar multiple $m = w_i m'$ with $i = \kappa$ or $1\leq i \leq \kappa-t$. }
\begin{proof}
We prove the equivalent following statement $P(t)$ by induction on $1 \leq t \leq N$:
\begin{displaymath}
P(t):\quad V_{t}^- \cap (\sum_{\kappa-t<j<\kappa} w_j \otimes \BW^{\otimes t-1}) = 0.
\end{displaymath}
For $t=1$, this is obvious. $P(2)$ comes from Example \ref{example: tensor square}. Assume that $t > 2$. Suppose that LHS of $P(t)$ contains a non-zero vector $y = \sum_{j=\kappa-t+1}^{\kappa-1}w_j \otimes x_j$. By Lemma \ref{lem: negative fundamental module}
\begin{displaymath}
V_{t}^- = U_q(\Glie) (w_1^{\otimes t-1} \otimes w_1) \subseteq U_q(\Glie) w_1^{\otimes t-1} \otimes \BW = V_{t-1}^- \otimes \BW. 
\end{displaymath}
The induction hypothesis $P(t-1)$ implies that $x_{\kappa-t+1} \neq 0$. A careful analysis of the first two tensor factors of $y$ in view of $V_{t}^- \subseteq V_{2}^- \otimes V_{t-2}^-$ leads to
\begin{displaymath}
y = w_{\kappa-t+1}\otimes \sum_{j=\kappa-t+2}^{\kappa-1} w_j \otimes y_j + \sum_{j=\kappa-t+2}^{\kappa-1} w_j \otimes (-q) w_{\kappa-t+1} \otimes y_j.
\end{displaymath}
Since $V_{t}^- \subseteq \BW \otimes V_{t-1}^-$, we must have $0\neq x_{\kappa-t+1} = \sum_{j=\kappa-t+2}^{\kappa-1} w_j \otimes y_j \in V_{t-1}^-$, in contradiction with $P(t-1)$. This proves $P(t)$.
\end{proof} 
Let us determine the components $\BV^{\otimes s} \otimes w_i \otimes \BW^{\otimes t-1}$ in $x_{a,b}$. By Equation \eqref{equ: fusion even odd},
\begin{displaymath}
R_{a,b}^{s|t} = F_t F_{t-1} \cdots F_2 F_1,\quad F_j = \prod_{l=1}^s (R_{aq^{-2l},bq^{2j}}^{+-})_{l,s+j}.
\end{displaymath}
Let $\tau \in \mathfrak{S}_t$. If $\tau(1) \neq t$, then $F_1$ fixes the term $v_{\kappa}^{\otimes s} \otimes w_{\kappa-t+\tau(1)} \otimes \cdots \otimes w_{\kappa-t+\tau(t)}$ in $v_{\kappa}^{\otimes s}\otimes w^{(t)}$. Applying $F_tF_{t-1}\cdots F_2$ to this term results in irrelevant components $\BV^{\otimes s} \otimes w_j \otimes \BW^{\otimes t-1}$ with $\kappa-t< j < \kappa$. We are reduced to consider the case $\tau(1) = t$ and to evaluate $R_{a,b}^{s|t}(v_{\kappa}^{\otimes s} \otimes w_{\kappa} \otimes x)$, where $x$ is a sum of $(t-1)$-fold tensor products of the $w_j$ with $\kappa-t < j < \kappa$. By Equation \eqref{for: positive-negative R matrix}, the term $\BV^{\otimes s}\otimes w_{\kappa} \otimes \BW^{\otimes t-1}$ in $F_1(v_{\kappa}^{\otimes s} \otimes w_{\kappa} \otimes x)$ is
\begin{displaymath}
\prod_{l=1}^{s} \frac{bq^2 q^{-1}-aq^{-2l} q}{bq^2-aq^{-2l}} v_{\kappa}^{\otimes s} \otimes w_{\kappa} \otimes x = q^{-s} \frac{bq^2 - a}{bq^{2}-aq^{-2s}} v_{\kappa}^{\otimes s} \otimes w_{\kappa} \otimes x.
\end{displaymath}
Notice that the $F_j$ with $2\leq j \leq t$ fix $v_{\kappa}^{\otimes s} \otimes w_{\kappa} \otimes x$. So the above term is exactly the component of $\BV^{\otimes s} \otimes w_{\kappa} \otimes \BW^{\otimes t-1}$ in $x_{a,b}$. For $1\leq i \leq \kappa-t$, again by Equation \eqref{for: positive-negative R matrix}, the terms $\BV^{\otimes s} \otimes w_i \otimes \BW^{\otimes t-1}$ in $F_1(v_{\kappa}^{\otimes s} \otimes w_{\kappa} \otimes x)$ and in $x_{a,b}$ are the same: 
\begin{displaymath}
\sum_{k=1}^s (-1)^{(s-k)(|i|+\odd)} \frac{bq^2(q_i-q_i^{-1})}{bq^2-aq^{-2k}} \prod_{l=k+1}^s \frac{bq^2 q^{-1}-aq^{-2l} q}{bq^2-aq^{-2l}} v_{\kappa}^{\otimes k-1} \otimes v_i \otimes v_{\kappa}^{\otimes s-k} \otimes w_i \otimes x.
\end{displaymath}
The coefficients are $(-1)^{(s-k)(|i|+\odd)}\frac{bq^{k-s+2}(q_i-q_i^{-1})}{bq^2-aq^{-2s}}$. Together with the claim, we conclude that the denominator of $x_{a,b}$ is $bq^2-aq^{-2s}$.
\end{proof}

The denominators of the first two rational functions are given as follows.
\begin{theorem}   \label{thm: even even poles}
Let $1 \leq s, t \leq M$. Let $u = \min(s,t)$. In the situation of Lemma \ref{lem: normalized r matrix even odd} (A), we have $\frac{X}{Y} \doteq \frac{N}{D}$ where $N = \prod_{j=1}^u (a - b q^{-2(t-u+j)}), D = \prod_{j=1}^u (a - b q^{2(s-u+j)})$.
Moreover, $N$ (resp. $D$) is the denominator of the $\End(V_s^+\otimes V_t^+)$-valued rational function $\frac{1}{X} R_{a,b}^{s,t}|_{V_s^+\otimes V_t^+}$ (resp. $\frac{1}{Y} R_{a,b}^{s,t}|_{V_s^+\otimes V_t^+}$).  \footnote{In the non-graded case, $U_q(\widehat{\mathfrak{gl}_M})$ has fundamental modules $V_{s,a}$ where $1\leq s < M$ and $a \in \BC^{\times}$ such that $\Pi(V_{s,a})$ is of the form $(1^{s-1}, q\frac{1-zaq^{-2s-2}}{1-zaq^{-2s}},1^{M-1-s})$; see \S \ref{subsec: classification}. Up to shifts of spectral parameters, $V_{s,a}$ is $V^{(s)}_a$ in \cite[\S 2.2]{Date}, and the denominator for $V_{a}^{(s)}$ and $V_{b}^{(t)}$ therein is a polynomial of degree $\min(s,t,M-s,M-t)$.\label{footnote:Okado} }
\end{theorem}
\begin{proof}
The idea is similar to that of Theorem \ref{thm: pole even odd}. We shall compute $\frac{X}{Y}$ and prove the statement for $\frac{1}{X} R_{a,b}^{s,t}|_{V_s^+\otimes V_t^+}$. Notice that $Y$ is computed in the proof of Lemma \ref{lem: normalized r matrix even odd}. 

\noindent \underline{Step I.} By definition $X$ is the coefficient of $(-q)^{l(\tau_0)} v^{(s)} \otimes v_t \otimes v_{t-1} \otimes \cdots \otimes v_1$ in $R_{a,b}^{s,t}(v^{(s)} \otimes v^{(t)})$; here $\tau_0 \in \mathfrak{S}_t$ is the permutation $j \mapsto t+1-j$.

\noindent {\it Claim 1.} For $1 \leq i,j \leq \kappa$, the term $\BV^{\otimes s} \otimes v_i$ appears in $R_{a,c}^{s,1}(v^{(s)} \otimes v_j)$ only if $i \leq j$.

This comes from the fact that $R_{a,c}^{s,1}(V_{s,a}^+ \otimes V_{1,c}^+) \subseteq V_{s,a}^+ \otimes V_{1,c}^+$ and $v^{(s)}$ is a highest $\ell$-weight vector of $V_{s,a}^+$. For $1\leq j \leq t$, let $X_j$ denote the coefficient of $v^{(s)} \otimes v_{j}$ in $R_{a,bq^{-2(t-j)}}^{s,1}(v^{(s)}\otimes v_{j})$. Then $X = X_tX_{t-1} \cdots X_1$ by Equation \eqref{equ: fusion even even}. 

If $j > s$, then the coefficient of $v_s \otimes v_{s-1} \otimes \cdots \otimes v_1 \otimes v_j$ in $R_{a,bq^{-2(t-j)}}^{s,1}(v_s \otimes v_{s-1}\otimes \cdots v_1 \otimes v_j)$ gives $X_j = \prod_{i=1}^s(aq^{-2i} - bq^{-2(t+1-j)})$. 

If $j \leq s$, then the coefficient of $u_1 := v_j \otimes (v_s \otimes v_{s-1} \otimes \cdots \otimes v_{j+1}) \otimes (v_{j-1} \otimes v_{j-2} \otimes \cdots v_1) \otimes v_j$ in $R_{a,bq^{-2(t-j)}}^{s,1}(u_1)$ gives $X_j = (aq^{-1} - bq^{-2(t+1-j)-1}) \times \prod_{i=2}^{s} (aq^{-2i} - bq^{-2(t+1-j)})$.

In the following, we mainly treat the case $s \leq t$ so that $u = s$. The case $s > t$ will be sketched at Step V. Consider the $\BV^{\otimes s+t}$-valued polynomial $u_2 := R_{a,b}^{s,t}(v^{(s)} \otimes v_{\kappa}^{\otimes t})$. As in the proof of Theorem \ref{thm: pole even odd}, the denominator of $\frac{1}{X} u_2$ is that of $\frac{1}{X} R_{a,b}^{s,t}|_{V_s^+\otimes V_t^+}$. Already $X \doteq \frac{a-bq^{-2t}}{a-bq^{-2(t-s)}}  \prod_{i=1}^s \prod_{j=1}^t (aq^{-2i}-bq^{-2j})$ and $\frac{Y}{X} \doteq \frac{D}{N}$.

\noindent \underline{Step II.} By Equations \eqref{for: Perk-Schultz matrix coefficients} and \eqref{equ: fusion even even}, $u_2$ is a linear combination of the $v_{j_1} \otimes \cdots \otimes v_{j_{s+t}}$ where $v_i$ appears once if $i \leq s$ and $t$ times if $i = \kappa$. Similar to the claim in the proof of Theorem \ref{thm: pole even odd}, it is enough to determine for a given pair $(0 \leq r \leq s,\sigma \in \mathfrak{S}_s)$ the coefficient $k_{r,\sigma}$ in $u_2$ of the vector $v_{\kappa}^{\otimes r} \otimes v_{\sigma(1)} \otimes v_{\sigma(2)} \otimes \cdots \otimes v_{\sigma(s)} \otimes v_{\kappa}^{\otimes t-r}$. 

\noindent \underline{Step III.} Let $1 \leq i \leq s$. Define $W_i$ to be the set of $\sigma \in \mathfrak{S}_s$ such that $\sigma(s) = i$. Set 
\begin{displaymath}
v_i^{(s)} := \sum_{\sigma \in W_i} (-q)^{l(\sigma)} v_{\sigma(1)} \otimes v_{\sigma(2)} \otimes \cdots \otimes v_{\sigma(s-1)} \in V_{s-1,a}^+.
\end{displaymath}
View $\mathfrak{S}_{s-1}$ as the subgroup of $\mathfrak{S}_s$ formed of permutations fixing $s$. The multiplication $\mathfrak{S}_{s-1} \longrightarrow  \mathfrak{S}_s, \sigma \mapsto \tau_i \tau_{i+1} \cdots \tau_{s-1} \sigma$ induces a bijective map $\mathfrak{S}_{s-1} \longrightarrow W_i$ which increases the length of permutations by $s-i$; here the $\tau_j := (j,j+1)$ denote simple transpositions.  Now the next two claims comes from Equations \eqref{for: Perk-Schultz matrix coefficients} and \eqref{equ: fusion even even}. 

\noindent {\it Claim 2.} $R_{a,b}^{s-1,t}(v_i^{(s)}\otimes v_{\kappa}^{\otimes t})$ is obtained from $(-q)^{s-i} R_{a,b}^{s-1,t}(v^{(s-1)} \otimes v_{\kappa}^{\otimes t})$ by replacing the $v_j$ in the tensor factors with $v_{j+1}$ whenever $i\leq j\leq s-1$.

\noindent {\it Claim 3.} The term $\BV^{\otimes s-1} \otimes v_i \otimes \BV^{\otimes t}$ in $u_2$ is obtained by inserting $f_s v_i$ at the $s$-th position of the tensor factors of  $R_{a,b}^{(s-1,t)}(v_i^{(s)} \otimes v_{\kappa}^{\otimes t})$. Here 
$$ (a).\quad f_s = \prod_{j=1}^t (aq^{-2s} - bq^{-2j}).$$

The next claim reduces the problem of $V_{s}^+\otimes V_t^+$ to the case $V_s^+\otimes V_s^+$.

\noindent {\it Claim 4.} Let $g_r$ be the coefficient of $v_{\kappa}^{\otimes r} \otimes v_1 \otimes v_2 \otimes \cdots \otimes v_r$ in $R_{a,b}^{r,r}(v^{(r)}\otimes v_{\kappa}^{\otimes r})$. Then for $\sigma \in \mathfrak{S}_r$, the coefficient of $v_{\kappa}^{\otimes r} \otimes v_{\sigma(1)} \otimes v_{\sigma(2)} \otimes \cdots \otimes v_{\sigma(r)} \otimes v_{\kappa}^{\otimes t-r}$ in $R_{a,b}^{r,t}(v^{(r)}\otimes v_{\kappa}^{\otimes t})$ is 
$$(b).\quad (-q)^{l(\sigma)} g_r \times \prod\limits_{i=1}^r \prod\limits_{j=r+1}^t (aq^{-2i-1}-bq^{-2j+1}) =: (-q)^{l(\sigma)} h_r.$$

Indeed, bases on the explicit formula of $v^{(r)} \in V_r^+$,  the coefficient of $v_{\kappa}^{\otimes r} \otimes v_{\sigma(1)} \otimes v_{\sigma(2)} \otimes \cdots \otimes v_{\sigma(r)}$ in $R_{a,b}^{r,r}(v^{(r)}\otimes v_{\kappa}^{\otimes r}) \in V_{r}^+ \otimes V_{r}^+$ should be $(-q)^{l(\sigma)} g_r$. Combining Claims 2--4, we obtain the following: for $0 \leq r \leq s$ and $\sigma \in \mathfrak{S}_s$, the coefficient in $u_2$ of the vector $v_{\kappa}^{\otimes r} \otimes v_{\sigma(1)} \otimes v_{\sigma(2)} \otimes \cdots \otimes v_{\sigma(s)} \otimes v_{\kappa}^{\otimes t-r}$  is  
$$(c). \quad (-q)^{x_{\sigma,s-r}} f_s f_{s-1} \cdots f_{r+1} h_{r} = k_{r,\sigma}$$
 for certain $x_{\sigma,l} \in \BZ_{\geq 0}$  defined inductively by Claims 2 and 4. 

\noindent \underline{Step IV.}  Compute $g_s$. Let $\sigma \in \mathfrak{S}_s$ with $i_l := \sigma^{-1}(l)$ for $1 \leq l \leq s$. By Equations \eqref{for: Perk-Schultz matrix coefficients} and \eqref{equ: fusion even even}, the coefficient of $v_{\kappa}^{\otimes s} \otimes v_1 \otimes v_2 \otimes \cdots \otimes v_s$ in $R_{a,b}^{s,s}(v_{\sigma(1)}\otimes v_{\sigma(2)}\otimes \cdots \otimes v_{\sigma(s)} \otimes v_{\kappa}^{\otimes s})$ is 
\begin{align*}
&\quad \quad (-1)^{l(\sigma)} \prod_{l\neq i_1}  (aq^{-2l} - bq^{-2}) \times aq^{-2i_1} (q-q^{-1}) \\
& \quad \times \prod_{l\neq i_1,i_2} (aq^{-2l} - bq^{-4}) \times aq^{-2i_2} (q-q^{-1}) \times (aq^{-2i_1-1} - bq^{-3})  \\
& \quad \times \prod_{l\neq i_1,i_2,i_3}(aq^{-2l} - bq^{-6}) \times aq^{-2i_3} (q-q^{-1}) \times (aq^{-2i_1-1} - bq^{-5}) (aq^{-2i_2-1} - bq^{-5}) \\
& \quad \times \cdots \times aq^{-2i_s} (q-q^{-1}) \times \prod_{l\neq i_s} (aq^{-2l-1}-bq^{-2s+1}) \\
&= (-1)^{l(\sigma)} a^s (q-q^{-1})^s q^{-s(s+1)-\frac{s(s-1)}{2}} \prod_{l=1}^s \prod_{j=1}^{s-1} (aq^{-2l} - bq^{-2j}).
\end{align*}
From the explicit formula of $v^{(s)}$ it follows that $g_s = a^s C_s \prod_{l=1}^s \prod_{j=1}^{s-1} (aq^{-2l} - bq^{-2j})$ where
\begin{displaymath}
  C_s := (q-q^{-1})^s q^{-s(s+1)-\frac{s(s-1)}{2}} (\sum_{\sigma \in \mathfrak{S}_s} q^{l(\sigma)}) = (q-q^{-1})^s q^{-s(s+1)-\frac{s(s-1)}{2}} \prod_{i=1}^s \frac{q^i - 1}{q-1} \neq 0. 
\end{displaymath}
Combining with the formulas (a)--(c) above, we have:
\begin{displaymath}
k_{r,\sigma} \doteq f_sf_{s-1}\cdots f_{r+1} h_{r} \doteq \frac{a^{r}}{\prod_{i=1}^{r}(aq^{-2i}-bq^{-2t})} \prod_{i=1}^s \prod_{j=1}^t (aq^{-2i}-bq^{-2j}).
\end{displaymath}
It follows that $\frac{k_{r,\sigma}}{X} \doteq a^{r} (a-bq^{-2(t-s)}) \prod_{j=t-r}^t (a-bq^{-2j})^{-1}$.

\noindent \underline{Step V.} Finally, let us consider the case $s > t$. We determine the p\^{o}les of the $\BV^{\otimes s+t}$-valued function  $R_{a,b}^{s,t}(v_{\kappa}^{\otimes s} \otimes v^{(t)}) =: u_3$. We have $X  \doteq \frac{a-bq^{-2t}}{a-b} \prod_{i=1}^s \prod_{j=1}^t (aq^{-2i}-bq^{-2j})$ and again $\frac{Y}{X} = \frac{D}{N}$. As in Step II, we are reduced to determine the coefficients $k_{r,\sigma}'$ in $u_3$ of the vectors $v_{\kappa}^{\otimes s-r} \otimes v_{\sigma(1)}\otimes \cdots v_{\sigma(t)} \otimes v_{\kappa}^{\otimes r}$ where $(0 \leq r \leq t, \sigma \in \mathfrak{S}_t)$. Similar arguments as Claims 2--4 and Step IV indicate that
\begin{displaymath}
k_{r,\sigma}' \doteq \frac{b^r}{\prod_{j=t-r+1}^t (aq^{-2}-bq^{-2j})} \prod_{i=1}^s\prod_{j=1}^t (aq^{-2i}-bq^{-2j}).
\end{displaymath}
Hence $\frac{k_{r,\sigma}'}{X} \doteq b^r (a-b) \prod_{j=t-r}^{t} (a-bq^{-2j})^{-1}$.
\end{proof}
\section{Weyl modules over quantum affine $\mathfrak{gl}(1,1)$}   \label{sec: Weyl modules}
In this section $M=N=1$ and $\Glie = \mathfrak{gl}(1,1)$. We discuss Weyl modules over $U_q(\Gaff)$, which were previously defined in \cite{Z1}.\footnote{In \cite[\S 4.1]{Z1} Weyl modules were defined in terms of Drinfeld loop generators. It is not difficult to translate it by RTT generators, using the Ding--Frenkel homomorphism reviewed in \cite[Theorem 3.12]{Z2}.}

Let $\CR_0$ be the set of rational functions $f(z) \in \BC(z)$ which are products of the $c\frac{1-za}{1-zac^2}$ with $a,c \in \BC^{\times}$. Let $\CR_1$ be the set of pairs $(f(z),P(z)) \in \CR_0 \times \BC[z]$ such that $P(z) \in 1+z\BC[z]$ and $\frac{P(z)}{f(z)} \in \BC[z]$. (So $\deg \frac{P}{f} = \deg P$.) For $(f,P) \in \CR_1$, the {\it Weyl module} $\CW(f;P)$ is the $U_q(\Gaff)$-module generated by an even vector $w$ and subject to relations: 
\begin{itemize}
\item[(W1)] $s_{12}(z) w = t_{12}(z) w = 0, s_{22}(z) w = t_{22}(z) w = w, s_{11}(z)w = f(z) w = t_{11}(z) w$;
\item[(W2)] $\frac{P(z)}{f(z)} s_{21}(z) w \in \CW(f;P)[[z]]$ is a polynomial of degree $\leq \deg P$.
\end{itemize}
In the last two equations of (W1), $f(z) \in \BC(z)$ is to be developed at $z = 0,\infty$ accordingly. Let $V(f)$ be the simple highest $\ell$-weight $U_q(\Gaff)$-module whose highest $\ell$-weight vector is even and verifies (W1). Let $f = \prod_{i} c_i \frac{1-za_i}{1-za_i c_i^2}$ be such that $a_i,c_i \in \BC^{\times}$ and $a_i \neq a_j c_j^2$ whenever $i \neq j$. Then $V(f) \cong \otimes_i V(c_i \frac{1-za_i}{1-za_i c_i^2})$ as $U_q(\Gaff)$-modules \cite[Theorem 5.2]{Z2}.

Recall from \cite[\S 2]{Z4} the Drinfeld generators of $U_q(\Gaff)$: $E_n,F_n,h_{s}, \phi_{n}^{\pm}, (s_{11}^{(0)})^{\pm 1}$ with $n \in \BZ$ and $s \in \BZ_{\neq 0}$. Set $F^+(z) := -\sum_{n>0}F_n z^n$ and $\phi^{\pm}(z) = \sum_{s\geq 0} \phi_{\pm s}^{\pm} z^{\pm s}$. We have:
\begin{align*}
& s_{11}^{(0)} \exp((q-q^{-1})\sum_{s>0} h_{s} z^s) = s_{11}(z),\quad (s_{11}^{(0)})^{-1} \exp((q^{-1}-q)\sum_{s<0} h_s z^s) = t_{11}(z), \\
& \Delta F^+(z) = \phi^+(z) \otimes F^+(z) + F^+(z) \otimes 1,\quad \Delta \phi^{\pm}(z) = \phi^{\pm}(z) \otimes \phi^{\pm}(z), \\
&[h_s,F_n] = \frac{1-q^{2s}}{s(q-q^{-1})}F_{n+s}, \quad F_nF_m + F_mF_n = 0 = h_s h_t - h_th_s,\quad F^+(z) = s_{21}(z)s_{11}(z)^{-1}.
\end{align*}
The $\phi_n^{\pm}$ are central. Let $U^-$ (resp. $U^{\geq 0}$) be the subalgebra generated by the $F_n$ (resp. the other Drinfeld generators). Then $U_q(\Gaff) = U^- U^{\geq 0}$.
\begin{lem} \label{lem: Weyl modules gl(1,1)}
For $(f,P) \in \CR_1$, $\CW(f;P)$ has a simple quotient $V(f)$ and is spanned by the $F_{n_1}F_{n_2}\cdots F_{n_s} w$ with $s \geq 0$ and $n_j > 0$ for $1\leq j \leq s$.  Furthermore, for all $w' \in \CW$, (W2) holds and $\phi^+(z) w' = f(z)^{-1}w'$.
\end{lem}
\begin{proof}
Let $V(f) \cong \otimes_{i=1}^lV(f_i)$ be a decomposition with $f_i = c_i \frac{1-za_i}{1-za_ic_i^2}$ for $1 \leq i \leq l$. Then $P(z)$ is divisible by $\prod_{i=1}^l (1-za_i)$. Let $x_i \in V(f_i)$ be a highest $\ell$-weight vector. Set $x := \otimes_{i=1}^l x_i$. From the coproduct of $F^+(z)$ and $\phi^+(z)$ we get 
\begin{displaymath}
F^+(z) x = \sum_{j=1}^l (\otimes_{i<j} \phi^+(z)x_i) \otimes F^+(z) x_j \otimes (\otimes_{i>j} x_i).
\end{displaymath}
In view of the explicit construction of $V(f_i)$ in \cite[\S 5]{Z4}, both $(1-za_i) \phi^+(z) x_i = c_i^{-1} (1-za_ic_i^2) x_i$ and $(1-za_i)F^+(z) x_i$ are polynomials of degree 1. It follows that $\prod_{i=1}^l (1-za_i) F^+(z) x$ is a polynomial of degree $\leq l$, implying (W2) for $x \in V(f)$. So $V(f)$ is a quotient of $\CW(f;P)$. Observe from \cite[\S 2]{Z4} that for the highest $\ell$-weight vector $x$ we have $\phi^+(z) x = s_{22}(z)s_{11}(z)^{-1}x \in \BC[[z]]x$. The remaining statements then come from $U_q(\Gaff) = U^- U^{\geq 0}$ and from the commuting relations of Drinfeld generators listed above.
\end{proof}
\begin{prop}   \label{prop: simple vs Weyl gl(1,1)}
Let $(f,P), (g,Q) \in \CR_1$. If the polynomials $\frac{P}{f}$ and $Q$ are co-prime, then $\CW(f;P) \otimes \CW(g;Q)$ is of highest $\ell$-weight and is a quotient of $\CW(fg;PQ)$.
\end{prop}
\begin{proof}
Let $\deg P = l$ and $\deg Q = u$. Let $x' \in \CW(f;P)$ and $y' \in \CW(g;Q)$ be homogeneous vectors. From the above lemma, $P(z) F^+(z) v' = \sum_{i=1}^l z^i x_i$ and $Q(z) F^+(z) y' = \sum_{j=1}^u z^j y_j$ for certain $x_i \in \CW(f;P)$ and $y_j \in \CW(g;Q)$.
\begin{align*}
F^+(z) (x'\otimes y') &= F^+(z) x' \otimes y' + (-1)^{|x'|} \phi^+(z) x' \otimes F^+(z) y' \\
&= \frac{1}{P(z)Q(z)} (Q(z) \sum_{i=1}^l z^i x_i \otimes y' + (-1)^{|x'|} \frac{P(z)}{f(z)} \sum_{j=1}^u z^j x' \otimes y_j).
\end{align*}
$P(z)Q(z)F^+(z) (x' \otimes y')$ is a polynomial of degree $\leq l+u$. Introduce $T(z) := (-1)^{|x'|} \frac{P(z)}{f(z)}$; it is a polynomial of degree $l$. Since $Q(z)$ and $T(z)$ are co-prime, the polynomials $z^i Q(z), z^j T(z)$ with $1\leq i \leq l$ and $1\leq j \leq u$ are linearly independent, and the $x_i \otimes y', x' \otimes y_j$ are in the subspace spanned by the coefficients of $P(z)Q(z)F^+(z) (x'\otimes y')$.

We have proved that $F_s x' \otimes y', x' \otimes F_sy' \in U^- (x'\otimes y')$ for all $s \in \BZ_{>0}$ and $x'\otimes y' \in \CW(f;P)\otimes \CW(g;Q)$. This implies by induction on $s,t \in \BZ_{\geq 0}$ that 
$$F_{n_1}F_{n_2}\cdots F_{n_s} x' \otimes F_{m_1}F_{m_2}\cdots F_{m_t} y' \in U^-(x'\otimes y')\quad \mathrm{for}\ n_1,\cdots,n_s,m_1,\cdots,m_t \in \BZ_{>0}.$$
Take $x',y'$ to be the highest $\ell$-weight generators of $\CW(f;P)$ and $\CW(g;Q)$. From the above lemma we see that 
$\CW(f;P)\otimes \CW(g;Q) = U^-(x'\otimes y')$. Moreover, $x'\otimes y'$ satisfies the conditions (W1)--(W2) in the definition of $\CW(fg;PQ)$.
\end{proof}
\section{Cyclicity of tensor products}   \label{sec: cyclicity} 
In this section we provide sufficient conditions for a tensor product of fundamental modules to be of highest $\ell$-weight, improving previously established ones in \cite{Z2}.

The main result of this section is as follows. For $r_1,r_2 \in \BZ_{\geq 0}$, set 
$$\Sigma(r_1,r_2) := \{q^{2l}\in \BC\ |\  r_2- \min(r_1,r_2) < l \leq r_2 \}.$$
\begin{theorem}   \label{thm: main result}
Let $k,l \in \BZ_{>0}$. Let $1 \leq r_1,r_2,\cdots,r_k \leq M$ and $1 \leq s_1,s_2,\cdots,s_l \leq N$. Let $a_1,a_2,\cdots,a_k, b_1,b_2,\cdots,b_l \in \BC^{\times}$. The $U_q(\Gaff)$-module $(\otimes_{i=1}^k V_{r_i,a_i}^+) \otimes (\otimes_{j=1}^l V_{s_j,b_j}^-) =: S$ is of highest $\ell$-weight if the following three conditions are satisfied:
\begin{itemize}
\item[(C1)] $\frac{a_j}{a_i} \notin \Sigma(r_i,r_j)$ for $1 \leq i < j \leq k$;
\item[(C2)] $\frac{b_i}{b_j} \notin \Sigma(s_i,s_j)$ for $1 \leq i < j \leq l$;
\item[(C3)] $a_i q^{-2r_i} \notin b_j q^2$ for $1\leq i\leq k$ and $1 \leq j \leq l$.
\end{itemize}
\end{theorem}
The proof of the theorem needs a series of reduction lemmas. In the following, for $V_1,V_2$ two $U_q(\Gaff)$-modules, we write $V_1 \simeq V_2$ if as $U_q(\Gaff)$-modules $V_1 \cong V_2 \otimes D$ for a one-dimensional $U_q(\Gaff)$-module $D$. Let $A,B$ be two Hopf superalgebras. Let $g: A \longrightarrow B$ be a morphism of superalgebras. (In general $g$ does not respect coproduct structures.) Let $V$ be a $B$-module and $W$ a sub-vector-superspace of $V$. Suppose that $W$ is stable by $g(A)$. The action of $g(A)$ endows $W$ with an $A$-module structure, denoted by $g^{\bullet} W$. 

From now on, set $U := U_q(\widehat{\mathfrak{gl}(1,1)})$. Let $g_1: U \longrightarrow U_q(\Gaff)$ be the superalgebra morphism defined by $\begin{pmatrix}
s_{11}(z) & s_{12}(z) \\
s_{21}(z) & s_{22}(z)
\end{pmatrix} \mapsto \begin{pmatrix}
s_{11}(z) & s_{1\kappa}(z) \\
s_{\kappa 1}(z) & s_{\kappa \kappa}(z)
\end{pmatrix} $ and similar formulas for the $t_{ij}(z)$. The following special property of {\it fundamental modules} is crucial in our reduction arguments. It was used implicitly in the proof of \cite[Theorem 4.2]{Z2}. We think of the trivial module also as fundamental modules: $\BC = V_{0,a}^{\pm}$.
\begin{lem}  \label{lem: crucial even fundamental}
Let $X_1$ (resp. $X_3$) be a positive (resp. negative) fundamental module with $x_1$ (resp. $x_3$) a lowest $\ell$-weight vector and let $Y_j := g_1(U) x_j \subseteq X_j$ for $j = 1,3$. Let $X_2$ be another $U_q(\Gaff)$-module and $Y_2$ a sub-vector-superspace of $X_2$ stable by $g_1(U)$. Then:
\begin{itemize}
\item[(1)] $Y_1 \otimes Y_2 \otimes Y_3$, as a subspace of the $U_q(\Gaff)$-module $X_1 \otimes X_2 \otimes X_3$, is stable by $g_1(U)$;
\item[(2)] The identity map $\mathrm{Id}: g_1^{\bullet}(Y_1 \otimes Y_2 \otimes Y_3) \cong g_1^{\bullet} Y_1 \otimes g_1^{\bullet} Y_2 \otimes g_1^{\bullet}Y_3$ is $U$-linear.
\end{itemize}
 \end{lem}
\begin{proof}
Let us assume first that $X_3$ is the trivial module. We shall prove that $s_{ij}(z) Y_1 = 0 = t_{ij}(z) Y_1$ whenever $i \in \{1,\kappa\}$ and $j \notin \{1,\kappa\}$; this will imply (1) and that the operators $g_1^{\otimes 2} (\Delta_U (y))$ and $\Delta_{U_q(\Gaff)} (g_1(y))$ on $Y_1 \otimes Y_2$ are identical for $y$ an arbitrary RTT generator of $U$, which proves (2). Let $X_1 = V_{s,a}^+$ with $1 \leq s \leq M$. By Lemma \ref{lem: fundamental modules} and Theorem \ref{thm: BKK Schur-Weyl duality}, $Y_1$ is two-dimensional, and its weights are $\lambda_1 := s \epsilon_{\kappa}, \lambda_2 := \epsilon_1 + (s-1) \epsilon_{\kappa}$. Let $u \in Y_1$ be of weight $\lambda_k$ with $k = 1,2$. Then $s_{ij}(z) u$ and $t_{ij}(z) u$ are of weight $\lambda_k + \epsilon_i - \epsilon_j$, which is not a weight of $V_{s,a}^+$ by Theorem \ref{thm: BKK Schur-Weyl duality}. So $s_{ij}(z) u = 0 = t_{ij}(z) u$, as desired.

Secondly let us assume that $X_1$ is the trivial module. We prove that $s_{ij}(z) Y_3 = 0 = t_{ij}(z) Y_3$ whenever $j \in \{1,\kappa\}$ and $i \notin \{1,\kappa\}$; this will also imply (2). Let $X_3 = V_{r,a}^-$ with $1\leq r \leq N$. The proof of Lemma \ref{lem: negative fundamental module} and Theorem \ref{thm: BKK Schur-Weyl duality} show that: $Y_3$ is two-dimensional with weights $\mu_1 := -r\epsilon_1, \mu_2 := -(r-1)\epsilon_1-\epsilon_{\kappa}$. Now $\mu_k+\epsilon_i-\epsilon_j$ is not a weight of $V_{r,a}^-$ for $k = 1,2$, leading to the desired result.

The general case is just a combination of the above two cases.
\end{proof}
Let us fix three distinguished vectors of a positive fundamental module $V$ as follows: $v^1$ is a highest $\ell$-weight vector; $v^3$ is a lowest $\ell$-weight vector; $v^2 = s_{1\kappa}^{(0)}v^3$. For a negative fundamental module $W$, the three vectors $w^1,w^2,w^3$ are defined in the same way. We shall be in the situation of Theorem \ref{thm: main result}, and add sub-indexes to emphasize the fundamental modules; for example $v_i^1 \in V_{r_i,a_i}^+$ and $w^3_j \in V_{s_j,b_j}^-$. The following lemma comes from the proof of Lemma \ref{lem: crucial even fundamental}. It is the reduction from $\mathfrak{gl}(M,N)$ to $\mathfrak{gl}(1,1)$.
\begin{lem}   \label{lem: reduction lowest gl(1,1)}
Set $W_i := g_1(U) v_i^3 \subseteq V_{r_i,a_i}^+$ and $W_j' := g_1(U) w_j^3 \subseteq V_{s_j,b_j}^-$. Then $W_i = \BC v_i^2 + \BC v_i^3$ and $W_j' = \BC w_j^2 + \BC w_j^3$. As $U$-modules
\begin{displaymath}
g_1^{\bullet} W_i \simeq V(q^{r_i} \frac{1-za_iq^{-2r_i-2}}{1-za_iq^{-2}}),\quad g_1^{\bullet}W_j' \simeq V(q^{-s_j}\frac{1-zb_jq^{2s_j}}{1-zb_j}),
\end{displaymath}
and $v_i^2,w_j^2$ (resp. $v_i^3,w_j^3$) are highest (resp. lowest) $\ell$-weight vectors.
\end{lem}
\begin{proof}
We prove the negative case. Replace $V_{s,b}^-$ by $\ev_{bq^2}^*V_s^-$ according to Lemma \ref{lem: negative fundamental module}. From the second part of the proof of Lemma \ref{lem: crucial even fundamental} we see that: $w^2,w^3$ are of weights $(1-s)\epsilon_1-\epsilon_{\kappa}, -s\epsilon_1$ respectively; $W' = \BC w^2+\BC w^2$ and $w^2$ is a highest $\ell$-weight vector of $g_1^{\bullet}W'$; 
$$ s_{11}(z)s_{\kappa \kappa}(z)^{-1} w^2 = \frac{q^{1-s}-zbq^{2+s-1}}{q-zbq^{2-1}} w^2 = q^{-s}\frac{1-zbq^{2s}}{1-zb} w^2 = t_{11}(z)t_{\kappa \kappa}(z)^{-1} w^2.$$
This proves the second isomorphism in the lemma.
\end{proof}
Let $U_2 := U_q(\widehat{\mathfrak{gl}(M-1,N)})$ and $g_2: U_2 \longrightarrow U_q(\Gaff)$ be the superalgebra morphism defined by $s_{ij}(z) \mapsto s_{i+1,j+1}(z)$ and similar formula for the $t_{ij}(z)$. Let $V_{r,a}^{2\pm}$ denote the positive/negative fundamental modules over $U_2$, with $(+,1 \leq r < M)$ or $(-,1 \leq r \leq N)$. 
\begin{lem}   \label{lem: reduction highest M}
Set $K_i := g_2(U_2) v_i^1 \subseteq V_{r_i,a_i}^+$ and $K_j' := g_2(U_2) w_j^1 \subseteq V_{s_j,b_j}^-$. Then $v_i^2 \in K_i$ and 
\begin{displaymath}
g_2^{\bullet} K_i \simeq V_{r_i-1,a_iq^{-2}}^{2+}, \quad g_2^{\bullet} K_j' \simeq V_{s_j,b_j}^{2-}
\end{displaymath}
as $U_2$-modules, with $v_i^1$ and $v_i^2$ being highest and lowest $\ell$-weight vectors of $g_2^{\bullet}K_i$ respectively. Furthermore, $K := (\otimes_{i=1}^k K_i) \otimes (\otimes_{j=1}^l K_j')$ is stable by $g_2(U_2)$, and and identity map $\mathrm{Id}: g_2^{\bullet}(K) \cong (\otimes_{i=1}^k g_2^{\bullet}K_i) \otimes (\otimes_{j=1}^lg_2^{\bullet} K_j')$ is $U_2$-linear.
\end{lem}
Let $U_3 := U_q(\widehat{\mathfrak{gl}(M,N-1)})$ and $g_3: U_3 \longrightarrow U_q(\Gaff)$ be the superalgebra morphism defined by $s_{ij}(z) \mapsto s_{ij}(z), t_{ij}(z) \mapsto t_{ij}(z)$. Let $V_{r,a}^{3\pm}$ denote the positive/negative fundamental modules over $U_3$, with $(+,1 \leq r \leq M)$ or $(-,1 \leq r < N)$. 
\begin{lem}   \label{lem: reduction highest N}
Set $L_i := g_3(U_3) v_i^1 \subseteq V_{r_i,a_i}^+$ and $L_j' := g_3(U_3) w_j^1 \subseteq V_{s_j,b_j}^-$. Then $w_j^2 \in L_j'$ and 
\begin{displaymath}
g_3^{\bullet} L_i \simeq V_{r_i,a_i}^{3+}, \quad g_3^{\bullet} L_j' \simeq V_{s_j-1,b_j}^{3-}
\end{displaymath}
as $U_3$-modules, with $w_j^1$ and $w_j^2$ being highest and lowest $\ell$-weight vectors of $g_3^{\bullet}L_j'$ respectively. Furthermore, $L := (\otimes_{i=1}^k L_i) \otimes (\otimes_{j=1}^l L_j')$ is stable by $g_3(U_3)$, and the identity map $\mathrm{Id}: g_3^{\bullet}(L) \cong (\otimes_{i=1}^k g_3^{\bullet}L_i) \otimes (\otimes_{j=1}^lg_3^{\bullet} L_j')$ is $U_3$-linear.
\end{lem}
Lemmas \ref{lem: reduction highest M} and \ref{lem: reduction highest N} can be deduced from Theorem \ref{thm: BKK Schur-Weyl duality} and Lemmas \ref{lem: fundamental modules}--\ref{lem: negative fundamental module}. We have used the above three reductions in \cite{Z2} to prove a weaker version of Theorem \ref{thm: main result}. The following lemma is new and is a {\it crucial} step in the proof of Theorem \ref{thm: main result}.
\begin{lem}  \label{lem: reduction Weyl gl(1,1)}
Set $W := g_1(U) ((\otimes_{i=1}^k v_i^1) \otimes (\otimes_{j=1}^{l}w_j^1)) \subseteq (\otimes_{i=1}^k V_{r_i,a_i}^+) \otimes (\otimes_{j=1}^l V_{s_j,b_j}^-)$. There exists a one-dimensional $U$-module $D$ making $W$ a quotient of 
\begin{displaymath}
D \otimes \CW(\prod_{i=1}^k \frac{q-za_iq^{-2r_i-1}}{1-za_iq^{-2r_i}} \times \prod_{j=1}^l \frac{1-zb_jq^2}{q-zb_jq}; \prod_{i=1}^k(1-za_iq^{-2r_i-2}) \times \prod_{j=1}^l (1-zb_jq^2)).
\end{displaymath}
\end{lem}
\begin{proof}
We can replace $V_{r,a}^+$ and $V_{s,b}^-$ by $\ev_{aq^{-2r}}^* V_r^+$ and $\ev_{bq^2}^*V_s^-$ respectively. By Lemmas \ref{lem: fundamental modules}--\ref{lem: negative fundamental module}, $v^1_i,w_j^1$ are of weights $\epsilon_1+\epsilon_2+ \cdots + \epsilon_{r_i},- \epsilon_{\kappa} - \epsilon_{\kappa-1} - \cdots - \epsilon_{\kappa-s_j+1}$. Let $v$ be ordered tensor product of the $v_i^1,w_j^1$ and $W = g_1(U) v$. By Equation \eqref{equ: evaluation morphism}:  $s_{1\kappa}(z) v = 0 = t_{1\kappa}(z)v$;  $s_{\kappa 1}(z) v$ is a polynomial of degree $\leq k+l$; $s_{ii}(z)v = f_i(z)v$ and $t_{ii}(z)v = g_i(z) v$ where
\begin{align*}
& f_1 = \prod_{i=1}^k (q-za_iq^{-2r_i-1}) \times \prod_{j=1}^l (1-zb_jq^2),\quad g_1 = f_1 \times (-z^{-1})^{k+l} \prod_{i=1}^k (a_i^{-1}q^{2r_i}) \prod_{j=1}^l (b_j^{-1}q^{-2}),  \\
& f_{\kappa} = \prod_{i=1}^k (1-za_iq^{-2r_i}) \times \prod_{j=1}^l (q-zb_jq), \quad g_{\kappa} = f_{\kappa} \times (-z^{-1})^{k+l} \prod_{i=1}^k (a_i^{-1}q^{2r_i}) \prod_{j=1}^l (b_j^{-1}q^{-2}).
\end{align*}
Let $D_1 = \phi_{[f_{\kappa}^{-1},g_{\kappa}^{-1}]}^*(\BC) = \BC d$ be the one-dimensional $U$-module. Then the tensor product of $U$-modules $D_1\otimes W$ is generated by $d \otimes v$. Moreover, $d\otimes v$ satisfies all the relations in the definition of $\CW(\frac{f_1}{f_{\kappa}}; q^{-k}f_1)$; the latter therefore has $D_1\otimes W$ as a quotient.
\end{proof}
\begin{cor}    \label{cor: conditions D E}
We have $v_1^3 \otimes (\otimes_{i=2}^k v_i^1) \otimes (\otimes_{j=1}^l w_j^1) \in g_1(U)(v_1^2 \otimes (\otimes_{i=2}^k v_i^1) \otimes (\otimes_{j=1}^l w_j^1))$ if $a_1 \neq a_i q^{-2r_i}, b_jq^4$ for all $i,j$. Similarly, if $b_l \neq a_iq^{-2r_i-2s_l}, b_jq^{-2s_l}$ for all $i,j$, then $(\otimes_{i=1}^k v_i^1) \otimes (\otimes_{j=1}^{l-1}w_j^1) \otimes w_l^3 \in g_1(U) ((\otimes_{i=1}^k v_i^1) \otimes (\otimes_{j=1}^{l-1}w_j^1) \otimes w_l^2)$.
\end{cor}
\begin{proof}
Let us prove the second part, the first part being similar. We are in the situation of Lemma \ref{lem: crucial even fundamental} where $X_1$ is trivial, $X_3 := V_{s_l,b_l}^-, Y_3 := W_l' \subseteq X_{3}$ (see Lemma \ref{lem: reduction lowest gl(1,1)}) and
$$X_2 := (\otimes_{i=1}^k V_{r_i,a_i}^+) \otimes (\otimes_{j=1}^{l-1} V_{s_j,b_j}^-), \quad Y_2 := g_1(U)((\otimes_{i=1}^k v_i^1) \otimes (\otimes_{j=1}^{l-1}w_j^1)) \subseteq  X_2.  $$
It follows that $g_1^{\bullet}(Y_2\otimes Y_3) = g_1^{\bullet}Y_2 \otimes g_1^{\bullet}Y_3$. The $U$-modules $g_1^{\bullet}(Y_2)$ and $g_1^{\bullet}(Y_3)$ (see Lemma \ref{lem: reduction lowest gl(1,1)}) are generated by highest $\ell$-weight vectors $(\otimes_{i=1}^k v_i^1) \otimes (\otimes_{j=1}^{l-1}w_j^1)$ and $w_l^2$ respectively.
It is therefore enough to prove that $g_1^{\bullet}Y_2 \otimes g_1^{\bullet}Y_3$ is of highest $\ell$-weight. Up to tensor products by one-dimensional modules, by Lemma \ref{lem: reduction Weyl gl(1,1)}, $g_1^{\bullet} Y_2$ is  a quotient of the Weyl module
$$\CW_2:=  \CW(\prod_{i=1}^k \frac{q-za_iq^{-2r_i-1}}{1-za_iq^{-2r_i}} \times \prod_{j=1}^l \frac{1-zb_jq^2}{q-zb_jq}; \prod_{i=1}^k(1-za_iq^{-2r_i-2}) \times \prod_{j=1}^{l-1} (1-zb_jq^2)); $$
by Lemmas \ref{lem: reduction lowest gl(1,1)} and \ref{lem: Weyl modules gl(1,1)} $g_1^{\bullet} Y_3$ is a quotient of the Weyl module
$$\CW_3 := \CW(q^{-s_l}\frac{1-zb_lq^{2s_l}}{1-zb_l}; 1-zb_lq^{2s_l}). $$ 
By assumption $b_lq^{2s_l} \notin \{ b_j, a_iq^{-2r_i} : 1\leq i \leq k, 1\leq j < l\}$. We deduce from Proposition \ref{prop: simple vs Weyl gl(1,1)} that $\CW_2 \otimes \CW_3$ is of highest $\ell$-weight. Its quotient $g_1^{\bullet}Y_2 \otimes g_1^{\bullet}Y_3$ is also of highest $\ell$-weight.
\end{proof}

Now we can prove three special cases of Theorem \ref{thm: main result}.  
\begin{cor}  \label{cor: tensor product even fundamental}
Under conditions (1)--(2) in Theorem \ref{thm: main result}, the tensor products $\otimes_{i=1}^k V_{r_i,a_i}^+$ and $\otimes_{j=1}^l V_{s_j,b_j}^-$ are both of highest $\ell$-weight.
\end{cor}
\begin{proof}
We shall prove the positive case by induction on $M$ and $k$; the negative case uses essentially the same arguments. It is useful to include the case $M = 0$ where  the $r_i = 0$ and $S$ is trivial. Let $M>0$. (C1) implies the conditions of the $a_i$ in the above corollary. So
\begin{displaymath}
(1):\quad v_1^3 \otimes (\otimes_{i=2}^k v_i^1) \in g_1(U) (v_1^2 \otimes (\otimes_{i=2}^k v_i^1)).
\end{displaymath}
Now consider the $U_2$-module $g_2^{\bullet}(K)$ in Lemma \ref{lem: reduction highest M} with $K = g_2(U_2) (\otimes_{i=1}^k v_i^1)$ (so $l = 0$). Since (C1) stays the same when replacing $M$ by $M-1$, the induction hypothesis applied to $M-1$ indicates that $g_2^{\bullet}(K)$ is of highest $\ell$-weight and
\begin{displaymath}
(2):\quad v_1^2 \otimes (\otimes_{i=2}^k v_i^1)  \in g_2(U_2) (\otimes_{i=1}^k v_i^1).
\end{displaymath}
Next the induction hypothesis applied to $k-1$ together with Lemma \ref{lem: cyclicity Chari} indicates that
\begin{displaymath}
(3): \quad  U_q(\Gaff)(v_1^3 \otimes (\otimes_{i=2}^k v_i^1))  = \otimes_{i=1}^k V_{r_i,a_i}^+.
\end{displaymath}
From (1)--(3) it follows that $\otimes_{i=1}^k V_{r_i,a_i}^+$ is of highest $\ell$-weight.
\end{proof}
\begin{cor}  \label{cor: mixed fundamental modules}
Let $1\leq r \leq M, 1\leq s \leq N$ and $a,b \in \BC^{\times}$. If $aq^{-2r} \neq bq^2$ then $V_{r,a}^+ \otimes V_{s,b}^-$ is of highest $\ell$-weight.
\end{cor}
\begin{proof}
Firstly use induction on $M$. By Corollary \ref{cor: conditions D E}, $v^3 \otimes w^1 \in g_1(U)(v^2\otimes w^1)$ if $a \neq bq^{4}$. Consider $K = g_2(U_2)v^1$ and $K' = g_2(U_2)w^1$ in Lemma \ref{lem: reduction highest M}; the induction hypothesis applied to $M-1$ shows that if $aq^{-2r} \neq bq^2$ then $v^2\otimes w^1\in g_2(U_2)(v^1\otimes w^1)$. Combining with Lemma \ref{lem: cyclicity Chari}, we see that $V_{r,a}^+ \otimes V_{s,b}^-$ is of highest $\ell$-weight if $a \neq bq^4$ and $aq^{-2r} \neq bq^2$. 

Secondly use induction on $N$. By Corollary \ref{cor: conditions D E}, $v^1\otimes w^3 \in g_1(U)(v^1\otimes w^2)$ if $b \neq aq^{-2r-2s}$. Consider $L = g_3(U_3)v^1$ and $L' = g_3(U_3)w^1$ in Lemma \ref{lem: reduction highest N}; the induction hypothesis applied to $N-1$ shows that if $aq^{-2r} \neq bq^{2}$ then $v^1\otimes w^2\in g_3(U_3) (v^1\otimes w^1)$. Thus $V_{r,a}^+\otimes V_{s,b}^-$ is of highest $\ell$-weight if $b \neq aq^{-2r-2s}$ and $aq^{-2r} \neq bq^2$.

Conclude as $\{a\neq bq^4, aq^{-2r} \neq b q^2 \} \cup \{b\neq aq^{-2r-2s}, aq^{-2r}\neq bq^2 \} = \{aq^{-2r} \neq bq^2\}$.
\end{proof}
\begin{cor}   \label{cor: simple fund modules two}
If $\frac{a_2}{a_1} \notin \Sigma(r_1,r_2)$ and $\frac{a_1}{a_2} \notin \Sigma(r_2,r_1)$, then $V_{r_1,a_1}^+ \otimes V_{r_2,a_2}^+ \cong V_{r_2,a_2}^+ \otimes V_{r_1,a_1}^+$ as $U_q(\Gaff)$-modules. Similarly, the $U_q(\Gaff)$-modules $V_{s_1,b_1}^- \otimes V_{s_2,b_2}^-$ and $V_{s_2,b_2}^- \otimes V_{s_1,b_1}^-$ are isomorphic if $\frac{b_1}{b_2} \notin \Sigma(s_1,s_2)$ and $\frac{b_2}{b_1} \notin \Sigma(s_2,s_1)$.
\end{cor}
\begin{proof}
It is enough to consider positive fundamental modules as $f^*(V_{r,a}^-)  \cong V_{r,aq^{2r+2}}'^{+}$ in view of the proof of Lemma \ref{lem: negative fundamental module}. By Lemma \ref{lem: normalized r matrix even odd} (A) and Theorem \ref{thm: even even poles}, there exists a $U_q(\Gaff)$-linear map $R_{r_1,r_2}: V_{r_1,a_1}^+ \otimes V_{r_2,a_2}^+ \longrightarrow V_{r_2,a_2}^+ \otimes V_{r_1,a_1}^+$ sending $v_1^1 \otimes v_2^1$ to $v_2^1 \otimes v^1_1$. By Corollary \ref{cor: tensor product even fundamental}, $v_2^1 \otimes v_1^1$ generates $V_{r_2,a_2}^+ \otimes V_{r_1,a_1}^+$.  So $R_{r_1,r_2}$ is an isomorphism of $U_q(\Gaff)$-modules.
\end{proof}

\begin{cor}   \label{cor: rearrangement of tensor products of even fund}
Under conditions (C1)--(C2) in Theorem \ref{thm: main result}, we have $\sigma \in \mathfrak{S}_k, \tau \in \mathfrak{S_l}$: 
\begin{itemize}
\item[(1)] as $U_q(\Gaff)$-modules $\otimes_{i=1}^k V_{r_i,a_i}^+ \cong \otimes_{i=1}^k V_{r_{\sigma(i)},a_{\sigma(i)}}^+$ and $\otimes_{j=1}^l V_{s_j,b_j}^- \cong \otimes_{j=1}^k V_{s_{\tau(j)},b_{\tau(j)}}^-$;
\item[(2)] $\frac{a_{\sigma(i)}}{a_{\sigma(j)}}, \frac{b_{\tau(j)}}{b_{\tau(i)}} \notin q^{2\BZ_{< 0}}$ whenever $i < j$.
\end{itemize}
\end{cor}
One can copy the proof of \cite[Corollary 2.2]{AK} using Corollaries \ref{cor: tensor product even fundamental} and \ref{cor: simple fund modules two}.

\noindent {\it Proof of Theorem \ref{thm: main result}.} \label{page: proof theorem} We use induction on $M+N$ and $k+l$. Owing to Corollary \ref{cor: rearrangement of tensor products of even fund}, we can assume that for $1 \leq i < j \leq k$, either $\frac{a_i}{a_j} \in q^{2\BZ_{\geq 0}}$ or $\frac{a_i}{a_j} \notin q^{2\BZ}$. If one of the two assumptions in Corollary \ref{cor: conditions D E} is satisfied, then using reduction to either $U_2$ in Lemma \ref{lem: reduction highest M}  or to $U_3$ in Lemma \ref{lem: reduction highest N} together with Lemma \ref{lem: cyclicity Chari}, we can conclude as in the proof of Corollary \ref{cor: tensor product even fundamental} that $S$ is of highest $\ell$-weight.

Suppose that first assumption in Corollary \ref{cor: conditions D E} fails, so that $a_1  = b_t q^4$ for some $1 \leq t \leq l$. Now let us rearrange the tensor product $\otimes_{j=1}^l V_{s_j,b_j}^-$ as in Corollary \ref{cor: rearrangement of tensor products of even fund} in such a way that $\frac{b_{\tau(l)}}{b_t} \in q^{2\BZ_{\geq 0}}$. (This is possible by Corollary \ref{cor: simple fund modules two}, and possibly $\tau(l) = t$.)  Assume next that the second assumption in Corollary \ref{cor: conditions D E} fails, so that  $b_{\tau(l)} = a_i q^{-2r_i-2s_{\tau(l)}}$ for some $1 \leq i \leq k$. It follows that $\frac{b_{\tau(l)}}{b_t} = \frac{a_i}{a_1} q^{4-2r_i-2s_{\tau(l)}}$. From $\frac{b_{\tau(l)}}{b_t} \in q^{2\BZ_{\geq 0}}$ we obtain $\frac{a_i}{a_1} \in q^{2\BZ}$ and so $\frac{a_i}{a_1} \in q^{2\BZ_{\leq 0}}$. This forces $r_i = s_{\tau(l)} = 1,a_i = a_1$ and $b_t q^2 = a_1q^{-2} = a_i q^{-2r_i}$, in contradiction with (C3). This completes the proof of Theorem \ref{thm: main result}. \hfill $\Box$

The following lemma (actually only the case $k+l=2$) will be used in the next section.
\begin{lem}   \label{lem: cyclicity tensor product two fund}
In Theorem \ref{thm: main result}, if $k = l = 1$, or $l = 0$ or $k = 0$, then (C3), or (C1) or (C2) is necessary for $S$ to be of highest $\ell$-weight.
\end{lem}
\begin{proof}
The case $k + l = 2$ comes from the proof of Lemma \ref{lem: negative fundamental module}, Lemma \ref{lem: normalized r matrix even odd} and Theorems \ref{thm: pole even odd}--\ref{thm: even even poles} and the case $kl = 0$ from Corollary \ref{cor: simple fund modules two} by induction on $k+l$ as in the proof of Theorem \ref{thm: main result} on page~\pageref{page: proof theorem}.   
\end{proof}
\section{Simplicity of tensor products}    \label{sec: simplicity}
We give equivalent conditions for a tensor product of fundamental modules to be simple. 

Let us recall the notion of {\it twisted dual} to pass from \lq\lq highest/lowest $\ell$-weight\rq\rq\ to \lq\lq simple\rq\rq. Let $V$ be a finite-dimensional $U_q(\Gaff)$-module. Its twisted dual, is the dual space $\hom(V,\BC)$ endowed with a $U_q(\Gaff)$-module structure, denoted by $V^{\vee}$, as follows: 
\begin{displaymath}
\langle x l, v \rangle := (-1)^{|l||x|} \langle l, \Sm \Psi(x) v \rangle\quad \mathrm{for}\ x \in U_q(\Gaff), l \in \hom(V,\BC), v \in V.  
\end{displaymath} 
For $V,W$ finite-dimensional $U_q(\Gaff)$-modules, by Equation \eqref{equ: transposition in quantum affine superalgebra} we have a natural isomorphism of $U_q(\Gaff)$-modules $(V\otimes W)^{\vee} \cong V^{\vee} \otimes W^{\vee}$. 
 For $1\leq i \leq n$, let $V_i$ be a finite-dimensional simple $U_q(\Gaff)$-module generated by a highest $\ell$-weight vector $v_i$. Then $V_i^{\vee}$ is again simple and contains a highest $\ell$-weight vector $v_i^*$ such that $v_i^*(v_i) = 1$. By duality argument, the tensor product $\otimes_{i=1}^n V_i$ is of highest $\ell$-weight if and only if the submodule $S$ of $\otimes_{i=1}^n V_i^{\vee}$ generated by $\otimes_{i=1}^n v_i^*$ is contained in all the other non-zero submodules. ($S$ must then be simple and is the {\it socle} of $\otimes_{i=1}^n V_i^{\vee}$.) The tensor product $\otimes_{i=1}^n V_i$ is simple if and only if $\otimes_{i=1}^nV_i$ and $\otimes_{i=1}^nV_i^{\vee}$ are both of highest $\ell$-weight. Similar statements hold for lowest $\ell$-weight modules.

  The twisted dual of $\BV(a)$ has been computed in \cite[Eq.(3.26)]{Z3}: 
$$ \BV(a)^{\vee} \simeq \BV(a^{-1}q^{2M-2N}). $$
Next let us compute the twisted dual of $\BW(a)$ in Example \ref{example: negative natural representation}. Denote by $\rho_a$ the representation of $U_q(\Gaff)$ on $\BW(a)$. As in \cite[\S 3.2]{Z3}, introduce 
\begin{align*}
X(z) &= (\rho_a \otimes \Id_{\End \BW}) (\sum_{i,j} s_{ij}(z) \otimes E_{ij})  \\
&= \sum_i (q_i^{-1}-zaq_i) E_{ii} \otimes E_{ii} + (1-za) \sum_{i \neq j} E_{ii} \otimes E_{jj} \\
&\quad + \sum_{i<j} (q_j^{-1}-q_j) E_{ji} \otimes E_{ij} + za\sum_{i>j} (q_j^{-1}-q_j) E_{ji} \otimes E_{ij} \in \End (\BW^{\otimes 2})[[z]].
\end{align*}
Set $A := (1-zaq^2)(1-zaq^{-2})$. By Equation \eqref{for: antipode}, we have
\begin{align*}
X(z)^{-1} &= (\rho_a \otimes \Id_{\End \BW}) (\sum_{i,j} \Sm(s_{ij}(z)) \otimes E_{ij})  \\
&= \frac{1}{A} \{ \sum_i (q_i-zaq_i^{-1}) E_{ii} \otimes E_{ii} + (1-za) \sum_{i \neq j} E_{ii} \otimes E_{jj} \\
&\quad \quad \quad + \sum_{i<j} (q_j-q_j^{-1}) E_{ji} \otimes E_{ij} + za\sum_{i>j} (q_j-q_j^{-1}) E_{ji} \otimes E_{ij}  \}.
\end{align*}
Similarly we can find the $\rho_a(\Sm(t_{ij}(z)))$. By comparing highest $\ell$-weights, we obtain
\begin{equation*}  
\BW(a)^{\vee} \simeq \BW(a^{-1}).
\end{equation*}

\noindent {\it Proof of Proposition \ref{prop: simplicity of fundamental modules}}.\label{page:simplicity} Let us consider the positive case; the negative case can then be implied as in the proof of Lemma \ref{lem: negative fundamental module}. Recall the following fact in \cite[Prop.4.7]{Z2}: $\otimes_{i=1}^s\BV(aq^{2i})$ is of lowest $\ell$-weight for all $a\in \BC^{\times}$ and $s \in \BZ_{>0}$. By taking twisted dual and using the formula of $\BV(a)^{\vee}$, we see that the lowest $\ell$-weight vector $v_{\kappa}^{\otimes s}$ generates the simple socle of $\otimes_{i=1}^s \BV(aq^{-2i})$ for all $a \in \BC^{\times}$. In particular, $V_{s,a}^+$ in Definition \ref{def: fundamental representations} is simple. \hfill $\Box$
\begin{lem} \label{lem: twisted dual of fundamental modules}
$(V_{r,a}^+)^{\vee} \simeq V_{r,a^{-1} q^{2(M-N+1+r)}}^+$ and $(V_{s,a}^-)^{\vee} \simeq V_{s,a^{-1}q^{-2s-2}}^-$ for $a \in \BC^{\times}$. 
\end{lem}
\begin{proof}
Let us prove the positive case. Recall from the proof of Proposition \ref{prop: simplicity of fundamental modules} that $\otimes_{j=r}^1\BV(aq^{-2j})$ is of lowest $\ell$-weight. This gives rise to a diagram of $U_q(\Gaff)$-modules
$$ \otimes_{j=r}^1\BV(aq^{-2j}) \xrightarrow{\theta} V_{r,a}^+ \xrightarrow{\tau} \otimes_{j=1}^r\BV(aq^{-2j}) $$
where $\theta,\tau$ are $U_q(\Gaff)$-linear and they both fix $v_{\kappa}^{\otimes r}$. Taking the twisted dual and using the formula of $\BV(a)^{\vee}$, one obtains a similar diagram where the $U_q(\Gaff)$-linear maps fix lowest $\ell$-weight vectors. One can use Definition \ref{def: fundamental representations} to conclude $(V_{r,a}^+)^{\vee} \simeq V_{r,a^{-1} q^{2(M-N+1+r)}}^+$.
\end{proof}
\noindent {\it Proof of Lemma \ref{lem: tensor product of odd even}}.\label{page:proof r-matrix} From Corollaries \ref{cor: tensor product even fundamental}--\ref{cor: mixed fundamental modules} we see that in Lemma \ref{lem: tensor product of odd even}, if the signature of $(V,W)$ is $(++)$ or $(--)$ or $(+-)$, then $V_a \otimes W_b$ is of highest $\ell$-weight for $\frac{a}{b}$ in the complementary of a finite subset of $\BC^{\times}$. By Lemma \ref{lem: twisted dual of fundamental modules} and the twisted dual argument, the same is true when replacing \lq\lq highest $\ell$-weight\rq\rq\ with \lq\lq simple\rq\rq. When $V_a \otimes W_b$ is simple, according to the highest $\ell$-weight classification in \S \ref{subsec: classification}, we must have a unique $U_q(\Gaff)$-linear isomorphism $V_a \otimes W_b \cong W_b \otimes V_a$ fixing highest $\ell$-weight vectors $v \otimes w \mapsto w \otimes v$. Such an isomorphism also resolves the case where the signature of $(V,W)$ is $(-+)$. \hfill $\Box$ 

We would like to emphasize that the above proofs of Proposition \ref{prop: simplicity of fundamental modules}, Lemmas \ref{lem: twisted dual of fundamental modules} and \ref{lem: tensor product of odd even} are independent of the results in \S\S \ref{sec: spin r matrix}--\ref{sec: even odd}. They use essentially Weyl modules in \S \ref{sec: Weyl modules}, \cite[Prop.4.7]{Z2} on lowest $\ell$-weight modules, and the twisted dual formula in \cite[Eq.(3.26)]{Z3}.
\begin{theorem}  \label{thm: simplicity of tensor products}
A tensor product of fundamental modules $V_1 \otimes V_2 \otimes \cdots \otimes V_s$ is simple if and only if so is $V_i \otimes V_j$ for all $1 \leq i < j \leq s$.
\end{theorem}
\begin{proof}
The \lq\lq only if\rq\rq\ part is trivial as in the non-graded case in \cite[\S 6]{H}: if $\otimes_{i=1}^s S_i$ is simple, then so are $S_i \otimes S_{i+1}$  and $\otimes_{j=1}^s S_{\sigma(j)}$ for $1\leq i < s$ and $\sigma \in \mathfrak{S}_s$ by comparing highest $\ell$-weights. For the \lq\lq if\rq\rq\ part, since $V_i \otimes V_j$ is simple, it is isomorphic to $V_j \otimes V_i$. Without loss of generality we can assume that $\otimes_{i=1}^n V_i =: S$ is of the form in Theorem \ref{thm: main result}: a tensor product of positive fundamental modules followed by that of negative fundamental modules. By Lemma \ref{lem: cyclicity tensor product two fund}, such a tensor product verifies the conditions (C1)--(C3) in Theorem \ref{thm: main result} and is therefore of highest $\ell$-weight.  Similar arguments adapted to $(\otimes_{i=1}^s V_i)^{\vee} \cong \otimes_{i=1}^s V_i^{\vee}$ by Lemma \ref{lem: twisted dual of fundamental modules}, we conclude that $S^{\vee}$ is of highest $\ell$-weight. So $S$ is simple. 
\end{proof}
\begin{rem}   \label{rem: simplicity condition explicit}
Let us make explicit the simplicity condition. Index $i = (r_i,\varepsilon_i)$ where $1\leq r_i \leq M$ if $\varepsilon_i = +$ and $1 \leq r_i \leq N$ if $\varepsilon_i = -$. Define
\begin{equation}    \label{equ: Delta}
\Delta_{ij} := \begin{cases}
\prod_{l=1}^{\min(r_i,r_j)}(a_i - a_j q^{\mp 2(r_j - \min(r_i,r_j)+l)}) & \mathrm{if}\ \varepsilon_i = \varepsilon_j = \pm,  \\
a_i - a_j q^{2r_i+2} & \mathrm{if}\ (\varepsilon_i,\varepsilon_j) = (+,-), \\
a_i - a_j q^{-2M+2N-2r_i-2} & \mathrm{if}\ (\varepsilon_i,\varepsilon_j) = (-,+).
\end{cases}
\end{equation}
Then $\otimes_{i=1}^s V_{r_i,a_i}^{\varepsilon_i}$ is simple if and only if $\Delta_{ij} \neq 0$ for all $i \neq j$.
\end{rem}
\begin{cor}  \label{cor: cyclicity simplicity of tensor products of even fundamental modules}
A tensor product of positive fundamental modules is simple if and only if it is both of highest $\ell$-weight and of lowest $\ell$-weight.
\end{cor}
\begin{proof}
The \lq\lq only if\rq\rq\ part is trivial by definition. The \lq\lq if\rq\rq\ part is a direct consequence of Theorem \ref{thm: even even poles}, Lemma \ref{lem: cyclicity tensor product two fund} and the above theorem.
\end{proof}
The above corollary remains true for tensor products of negative fundamental modules, by using the pull back $f^{*}$ in the proof of Lemma \ref{lem: negative fundamental module}. In \cite[\S 5]{Z2}, the above corollary was proved for all finite-dimensional simple modules over a Borel subalgebra of $U_q(\widehat{\mathfrak{gl}(1,1)})$ (and so over the full quantum affine superalgebra), the so-called $q$-Yangian.
\begin{example}  \label{example: lowest weight even odd}
Corollary \ref{cor: cyclicity simplicity of tensor products of even fundamental modules} fails if \lq\lq positive\rq\rq\ is removed. Consider $S := \BV(a) \otimes \BW(b) = V_{1,aq^2}^+ \otimes V_{1,bq^{-2}}^-$.  We have $v^1_1 = v^2_1 = v_1$ and $w^1_2 = w^2_2 = w_{\kappa}$. Set $W = g_1(U) v_1^1$ and $W' = g_1(U) w_2^1$. By Lemmas \ref{lem: crucial even fundamental}, \ref{lem: reduction lowest gl(1,1)} and \ref{lem: reduction Weyl gl(1,1)}, as $U$-modules
\begin{displaymath}
g_1^{\bullet}(W \otimes W') \cong g_1^{\bullet} W \otimes g_1^{\bullet} W' \simeq V(\frac{q-zaq^{-1}}{1-za}) \otimes V(\frac{1-zb}{q-zbq^{-1}}).
\end{displaymath}
It follows that $S$ is of highest $\ell$-weight and of lowest $\ell$-weight if $a \neq b$. On the other hand, $S$ is simple if and only if $a \neq b$ and $b \neq a q^{-2M+2N}$. 
\end{example}

\begin{example}
Let $1\leq s \leq M$ and $1 \leq t \leq N$ be such that $s-t = M-N$. The tensor product $V_{s,a}^+ \otimes V_{t,b}^-$ is simple if and only if it is of highest $\ell$-weight.
\end{example}

\section{Final remarks}   \label{sec: discussion}
In this final section, we make remarks which are not used in the proof of main results.

We use the convention in Remark \ref{rem: simplicity condition explicit}: associated to an index $1\leq i \leq s$ is a couple $(r_i,\varepsilon_i)$ where either $(\varepsilon_i = +, 1\leq r_i \leq M)$ or $(\varepsilon_i=-,1\leq r_i \leq N)$. Consider the tensor product $S := \otimes_{i=1}^s V_{r_i,a_i}^{\varepsilon_i}$. We want to know when $S$ is of highest $\ell$-weight.

Let us call $(\varepsilon_1\varepsilon_2\cdots\varepsilon_s)$ the signature of $S$. Theorem \ref{thm: main result} gives a criteria for $S$ to be of highest $\ell$-weight in signature $(++\cdots +--\cdots-)$. In the proof of Theorem \ref{thm: main result}, Corollary \ref{cor: conditions D E} is the crucial step to go from $s$ to $s-1$, whose proof relies on reductions: from $\Glie$ to $\mathfrak{gl}(M,N-1)$ in Lemma \ref{lem: reduction highest N}; from $\Glie$ to $\mathfrak{gl}(M-1,N)$ in Lemma \ref{lem: reduction highest M}; from $\Glie$ to $\mathfrak{gl}(1,1)$ in Lemmas \ref{lem: crucial even fundamental}, \ref{lem: reduction lowest gl(1,1)} and \ref{lem: reduction Weyl gl(1,1)}. Except Lemma \ref{lem: crucial even fundamental}, all the other lemmas hold regardless of the signature of $S$. For Lemma \ref{lem: crucial even fundamental}, we need the case $X_1$ negative and $X_3$ positive.

Let us define the quantum affine superalgebra $U_{q^{-1}}(\Gaff)$ in the same way as $U_q(\Gaff)$ except that $q$ is replaced by $q^{-1}$ everywhere; let $\widetilde{s}_{ij}^{(n)},\widetilde{t}_{ij}^{(n)}, \widetilde{s}_{ij}(z),\widetilde{t}_{ij}(z)$ denote its RTT generators.
\begin{lem}   \label{lem: involution}
There is an isomorphism of Hopf superalgebras $h: U_{q^{-1}}(\Gaff) \longrightarrow U_q(\Gaff)^{\mathrm{cop}}$ 
\begin{displaymath}
\sum_{i,j} \widetilde{s}_{ij}(z) \otimes E_{ij} \mapsto (\sum_{i,j} s_{ij}(z) \otimes E_{ij})^{-1}, \quad \sum_{i,j} \widetilde{t}_{ij}(z) \otimes E_{ij} \mapsto (\sum_{i,j} t_{ij}(z) \otimes E_{ij})^{-1}.
\end{displaymath}
\end{lem}
\begin{proof}
The idea is the same as that of \cite[Prop.3.4]{Z2}, based on the identity
\begin{displaymath}
R_q(z,w)^{-1} = \frac{1}{(zq-wq^{-1})(zq^{-1}-wq)} R_{q^{-1}}(z,w) \in \End(\BV^{\otimes 2})(z,w). 
\end{displaymath}
and on the definition of $U_q(\Gaff)$ in \cite[Definition 3.5]{Z2}.
\end{proof}
$h$ is inspired by the involution of quantum affine algebras in \cite[Appendix A]{AK}.

Let $\widetilde{V}_{r,a}^{\pm}$ be the corresponding fundamental modules over $U_{q^{-1}}(\Gaff)$. From the definition of twisted dual and from Lemma \ref{lem: twisted dual of fundamental modules}, we obtain
\begin{equation}    \label{equ: involution}
h^* V_{r,a}^+ \simeq \widetilde{V}_{r,aq^{-2M+2N-2r-2}}^+,\quad h^*V_{r,a}^- \simeq \widetilde{V}_{r,aq^{2r+2}}^-.
\end{equation}

Now Lemma \ref{lem: crucial even fundamental} and Corollary \ref{cor: conditions D E} can be generalized accordingly. For this purpose, let us define the $K_{ij}^l$ and $K_{ij}^r$ associated to the tensor product $S = \otimes_{i=1}^s V_{r_i,a_i}^{\varepsilon_i}$  as follows:
\begin{align}
& (K_{ij}^{l},K_{ij}^r) := \begin{cases}
(1,1) & \mathrm{if}\ \varepsilon_i = \varepsilon_j, \\
(a_i - a_j q^4, a_i - a_j q^{2r_i+2r_j}) & \mathrm{if}\ (\varepsilon_i,\varepsilon_j) = (+,-), \\
(a_i - a_j q^{-2M+2N-4},  a_i - a_j q^{-2M+2N-2r_i-2r_j}) & \mathrm{if}\ (\varepsilon_i,\varepsilon_j) = (-,+),
\end{cases}  \label{equ: K} \\
& f_i^l := \prod_{j<i} \Delta_{ij} \times \prod_{j\neq i} K_{ij}^l,\quad f_i^r := \prod_{j\neq i} K_{ji}^r \times \prod_{j>i} \Delta_{ji}. \label{equ: f}
\end{align} 
To unify notations in \S \ref{sec: cyclicity}, let $u_i^1$ and $u_i^3$ be highest and lowest $\ell$-weight vectors of $V_{r_i,a_i}^{\varepsilon_i}$ and $u_i^2 = s_{1\kappa}^{(0)}u^3_i$. From Theorem \ref{thm: BKK Schur-Weyl duality} and Lemmas \ref{lem: fundamental modules}--\ref{lem: negative fundamental module}, we see that $u_i^2 \doteq h(\widetilde{s}_{1\kappa}^{(0)}) u_i^3$. Corollary \ref{cor: conditions D E} together with its proof is now generalized as follows.
\begin{cor}
Assume that $\Delta_{ij} \neq 0$ for all $1 \leq i < j \leq s$. If $f_i^l \neq 0$, then $S \cong V_{r_i,a_i}^{\varepsilon_i} \otimes (\otimes_{j\neq i} V_{r_j,a_j}^{\varepsilon_j}) =: S_i^l$ as $U_q(\Gaff)$-modules and $u_i^3 \otimes (\otimes_{j\neq i} u_j^1) \in U_q(\Gaff) (u_i^2 \otimes (\otimes_{j\neq i} u_j^1)) \subseteq S_i^l$. Similarly, if $f_i^r \neq 0$, then $S \cong (\otimes_{j\neq i} V_{r_j,a_j}^{\varepsilon_j}) \otimes V_{r_i,a_i}^{\varepsilon_i} =: S_i^r$ as $U_q(\Gaff)$-modules and $(\otimes_{j\neq i}u_j^1) \otimes u_i^3 \in U_q(\Gaff)((\otimes_{j\neq i}u_j^1) \otimes u_i^2) \subseteq S_i^r$.
\end{cor}
In the corollary, $\otimes_{j\neq i}$ means the ordered tensor product $(\otimes_{j=1}^{i-1}) \otimes (\otimes_{j=i+1}^s)$. We arrive at the following problem of linear algebra. 

\begin{question}   \label{question: linear algebra}
For $1 \leq i \leq s$, let $(r_i,\varepsilon_i)$ be as above and let $a_i \in \BC^{\times}$. Define $\Delta_{ij}, f_i^l,f_i^r$ by Equations \eqref{equ: Delta}, \eqref{equ: K} and \eqref{equ: f}. Suppose that $f_i^l = f_i^r = 0$ for all $1 \leq i \leq s$. Then is it necessarily true that $\prod_{i<j} \Delta_{ij} = 0$?
\end{question}
\begin{rem}   \label{rem: idea of proof in general case}
Suppose that the answer to the above question is positive. We can argue as in the proof of Corollary \ref{cor: tensor product even fundamental} to conclude that $S = \otimes_{i=1}^s V_{r_i,a_i}^{\varepsilon_i}$ is of highest $\ell$-weight if $\prod_{i< j} \Delta_{ij} \neq 0$. 
\end{rem}
The proof of Theorem \ref{thm: main result} together with Lemma \ref{lem: involution} and Equation \eqref{equ: involution} actually affirms the cases $(\varepsilon_1\cdots\varepsilon_s) = (++\cdots + - - \cdots -)$ and $(--\cdots - ++ \cdots +)$. So Theorem \ref{thm: main result} remains true when the tensor product is of signature $(--\cdots - ++ \cdots +)$.
\begin{example}   \label{example: case s = 3}
Let $s = 3$. The answer to Question \ref{question: linear algebra} is affirmative. Indeed the only essential difficulty appears when $(\varepsilon_1\varepsilon_2\varepsilon_3) = (+-+)$ and $\prod_{i<j} \Delta_{ij} \neq 0 = \Delta_{21} = \Delta_{32}$. (If $\Delta_{21} \neq 0$ then one can exchange $\varepsilon_1$ and $\varepsilon_2$ to arrive at the known signature $(-++)$; similar arguments for $\Delta_{32}$.) Suppose $f_i^l = f_i^r= 0$. By definition, $f_1^l = K_{12}^l = 0$ and $f_3^r = K_{23}^r = 0$. From
\begin{displaymath}
K_{12}^l = 0 = \Delta_{21} = a_1 - a_2 q^4 = a_2 - a_1 q^{-2M+2N-2r_2-2}
\end{displaymath}
we get $2 = 2M-2N+2r_2$. Next from $K_{23}^r = a_2 - a_3 q^{-2M+2N-2r_2-2r_3} = 0$ we get $a_1 = a_3 q^{2-2r_3}$. But $\Delta_{13} \neq 0$, we have $\min(r_1,r_3) = 1$. Since $\Delta_{12} = a_1 - a_2 q^{2r_1+2} \neq 0$ and $a_1 = a_2 q^4$, $r_1 > 1$. Since $\Delta_{23} = a_2 - a_3q^{-2M+2N-2r_2-2} \neq 0$ and $a_2 = a_3 q^{-2M+2N-2r_2-2r_3}$, $r_3 > 1$. It follows that $\min(r_1,r_3) > 1$, a contradiction. As a consequence, $V_{r_1,a_1}^{\varepsilon_1} \otimes V_{r_2,a_2}^{\varepsilon_2} \otimes V_{r_3,a_3}^{\varepsilon_3}$ is of highest $\ell$-weight if $\prod_{i<j}\Delta_{ij} \neq 0$.
\end{example} 
\begin{example}   \label{example: case 1}
Assume that $r_i = 1$ for all $1 \leq i \leq s$. Since $f_1^l = 0 = \prod_{i=2}^s K_{1i}^l = 0$, there exists $1<i\leq s$ such that $K_{1i}^l = 0$. It follows that $\varepsilon_1 \neq \varepsilon_i$ and $\Delta_{1i} = K_{1i}^l = 0$. As a consequence, the tensor product $\otimes_{i=1}^s V_{1,a_i}^{\varepsilon_i}$ is of highest $\ell$-weight if $\prod_{i<j}\Delta_{ij} \neq 0$.
\end{example}

\end{document}